\documentclass[11pt]{amsart}

\DeclareMathOperator{\Lie}{Lie}\DeclareMathOperator{\im}{im}

\DeclareMathOperator{\ad}{ad}

\DeclareMathOperator{\Aut}{Aut}

\usepackage[usenames,dvipsnames]{xcolor}

\usepackage{amsmath,amssymb,amsthm,graphics,amscd}
\usepackage{lmodern}
\usepackage{longtable,pifont,listings,lstlang0}

\usepackage[T1]{fontenc}
\usepackage{textcomp}
\usepackage[%
  edge length = 1cm,
  root radius = 0.1cm
]{dynkin-diagrams}
\usepackage{cite}
\usepackage{color}
\usepackage{graphicx}
\usepackage{epsf}
\usepackage{fancyhdr}
 
\usepackage{pdflscape}
\usepackage{array,booktabs}
\newcolumntype{M}[1]{>{\centering\arraybackslash}m{#1}}

\usepackage{hyperref}
\hypersetup{
colorlinks=true,citecolor=Purple}

  \renewenvironment{thebibliography}[1]{
    \begin{oldthebibliography}{#1}
      \setlength{\parskip}{0ex}
      \setlength{\itemsep}{0ex}
  }
  {
    \end{oldthebibliography}
  }
  
\setlongtables
\begin{document}
\lstset{language=Magma, linewidth=0.8\textwidth, xleftmargin=10pt, framextopmargin=5pt, frame=lines,basicstyle=\ttfamily}

\newcommand{\cmark}{\ding{51}}%
\newcommand{\xmark}{\ding{55}}%

\newcounter{rownum}
\setcounter{rownum}{0}
\newcommand{\ab}{\addtocounter{rownum}{1}\arabic{rownum}}

\newcommand{\x}{$\times$}
\newcommand{\bb}{\mathbf}

\newcommand{\Ind}{\mathrm{Ind}}
\newcommand{\Char}{\mathrm{char}}
\newcommand{\hra}{\hookrightarrow}

\newtheorem{lemma}{Lemma}[section]
\newtheorem{theorem}[lemma]{Theorem}
\newtheorem*{claim}{Claim}
\newtheorem{cor}[lemma]{Corollary}
\newtheorem{conjecture}[lemma]{Conjecture}
\newtheorem{prop}[lemma]{Proposition}
\newtheorem{question}[lemma]{Question}

\theoremstyle{definition}
\newtheorem{example}[lemma]{Example}
\newtheorem{examples}[lemma]{Examples}
\newtheorem{algorithm}[lemma]{Algorithm}
\newtheorem*{algorithm*}{Algorithm}

\theoremstyle{remark}
\newtheorem{remark}[lemma]{Remark}
\newtheorem{remarks}[lemma]{Remarks}
\newtheorem{obs}[lemma]{Observation}

\theoremstyle{definition}
\newtheorem{defn}[lemma]{Definition}

  \def\hal{\unskip\nobreak\hfil\penalty50\hskip10pt\hbox{}\nobreak
  \hfill\vrule height 5pt width 6pt depth 1pt\par\vskip 2mm}

\newcommand{\SF}{\mathcal{S}}

\renewcommand{\theenumi}{\roman{enumi}}
\renewcommand{\labelenumi}{(\roman{enumi})}
\newcommand{\Hom}{\mathrm{Hom}}
\newcommand{\Int}{\mathrm{int}}
\newcommand{\Ext}{\mathrm{Ext}}
\newcommand{\opH}{\mathrm{H}}
\newcommand{\D}{\mathcal{D}}
\newcommand{\SO}{\mathrm{SO}}
\renewcommand{\O}{\mathrm{O}}
\newcommand{\Sp}{\mathrm{Sp}}
\newcommand{\SL}{\mathrm{SL}}
\newcommand{\PGL}{\mathrm{PGL}}
\newcommand{\PSp}{\mathrm{PSp}}
\newcommand{\Spin}{\mathrm{Spin}}
\newcommand{\HSpin}{\mathrm{HSpin}}
\newcommand{\GL}{\mathrm{GL}}
\newcommand{\OO}{\mathcal{O}}
\newcommand{\Y}{\mathbf{Y}}
\newcommand{\FF}{\mathcal{F}}
\newcommand{\X}{\mathbf{X}}
\newcommand{\diag}{\mathrm{diag}}
\newcommand{\End}{\mathrm{End}}
\newcommand{\tr}{\mathrm{tr}}
\newcommand{\Stab}{\mathrm{Stab}}
\newcommand{\red}{\mathrm{red}}

\renewcommand{\H}{\mathcal{H}}
\renewcommand{\u}{\mathfrak{u}}
\newcommand{\Ad}{\mathrm{Ad}}
\newcommand{\N}{\mathcal{N}}
\newcommand{\Z}{\mathbb{Z}}
\newcommand{\la}{\langle}\newcommand{\ra}{\rangle}
\newcommand{\gl}{\mathfrak{gl}}
\newcommand{\g}{\mathfrak{g}}
\newcommand{\s}{\mathfrak{s}}
\renewcommand{\o}{\mathfrak{o}}
\newcommand{\F}{\mathbb{F}}
\newcommand{\m}{\mathfrak{m}}
\renewcommand{\b}{\mathfrak{b}}
\newcommand{\p}{\mathfrak{p}}
\newcommand{\q}{\mathfrak{q}}
\renewcommand{\l}{\mathfrak{l}}
\newcommand{\del}{\partial}
\newcommand{\h}{\mathfrak{h}}
\renewcommand{\t}{\mathfrak{t}}
\renewcommand{\k}{\Bbbk}
\newcommand{\Gm}{\mathbb{G}_m}
\renewcommand{\c}{\mathfrak{c}}
\renewcommand{\r}{\mathfrak{r}}
\newcommand{\n}{\mathfrak{n}}
\renewcommand{\s}{\mathfrak{s}}
\newcommand{\Q}{\mathbb{Q}}
\providecommand{\C}{\mathbb{C}}
\newcommand{\z}{\mathfrak{z}}
\newcommand{\pso}{\mathfrak{pso}}
\newcommand{\pgl}{\mathfrak{pgl}}
\newcommand{\so}{\mathfrak{so}}
\renewcommand{\sl}{\mathfrak{sl}}
\newcommand{\psl}{\mathfrak{psl}}
\renewcommand{\sp}{\mathfrak{sp}}
\newcommand{\Ga}{\mathbb{G}_a}
\newcommand{\sh}{\mathsf{h}}
\newcommand{\e}{\mathsf{e}}
\newcommand{\A}{\mathsf{A}}
\newcommand{\B}{\mathsf{B}}
\newcommand{\CC}{\mathsf{C}}
\newcommand{\dd}{\mathrm{d}}

\newenvironment{changemargin}[1]{%
  \begin{list}{}{%
    \setlength{\topsep}{0pt}%
    \setlength{\topmargin}{#1}%
    \setlength{\listparindent}{\parindent}%
    \setlength{\itemindent}{\parindent}%
    \setlength{\parsep}{\parskip}%
  }%
  \item[]}{\end{list}}

\parindent=0pt
\addtolength{\parskip}{0.5\baselineskip}

\title{On extensions of the Jacobson--Morozov theorem to even characteristic}
\author{David I. Stewart}
\address{Department of Mathematics, Alan Turing Building, Manchester, M13 9PL, UK} 
\email{david.i.stewart@manchester.ac.uk}

\author{Adam R. Thomas} 
\address{Department of Mathematics, University of Warwick, Coventry, CV4 7AL, UK} 
\email{adam.r.thomas@warwick.ac.uk }
\pagestyle{plain}
\begin{abstract}Let $G$ be a simple algebraic group over an algebraically closed field $\k$ of characteristic $2$. We consider analogues of the Jacobson--Morozov theorem in this setting. More precisely, we classify those nilpotent elements with a simple $3$-dimensional Lie overalgebra in $\g:=\Lie(G)$ and also those with overalgebras isomorphic to the algebras $\Lie(\SL_2)$ and $\Lie(\PGL_2)$. This leads us to calculate the dimension of Lie automiser $\n_\g(\k\cdot e)/\c_\g(e)$ for all nilpotent orbits; in even characteristic this quantity is very sensitive to isogeny.
\end{abstract}
\maketitle

\section{Introduction}
The Jacobson--Morozov theorem states that every nilpotent element $e$ in the Lie algebra $\g$ of a semisimple complex algebraic group $G$ can be extended to an $\sl_2$-triple $(e,h,f)\in\g^3$ satisfying the famous commutator relations \[\tag{$\star$}[h,e]=2e,\ [h,f]=-2f,\ \text{and }[e,f]=h.\] In recent work \cite{ST18} building on previous efforts by Pommerening, Premet and Kostant---among others---the authors gave an exhaustive account of when the statement of the theorem holds for a reductive algebraic group $G$ over an algebraically closed field $\k$ of characteristic $p\geq 3$, with sharp statements about uniqueness of the embedding up to conjugacy. It turned out that \emph{all} nilpotent elements can be embedded into $\sl_2$-triples with essentially one exception: in $G_2$ when $p=3$, a nilpotent element in the exceptional orbit (usually labelled $\tilde A_1^{(3)}$) lives in no subalgebra isomorphic to $\sl_2$. 

As can be seen from the commutator relations above, there is some awkwardness in finding the right analogue of the statement of the Jacobson--Morozov theorem in characteristic $2$ and there seem to be  three reasonable approaches. Firstly, one may take the straight modulo $2$ reduction of the relations ($\star$) to get the Lie algebra $\sl_2=\Lie(\SL_2)$, in which $h$ is central, so that $\sl_2$ is no longer simple, and indeed the reduced elements $(\bar e,\bar h,\bar f)$ generate a (nilpotent) Heiseinberg algebra. Taking instead a minimal lattice with $[h,e]=e$, $[h,f]=f$ and $[e,f]=2h$, then the reduction modulo $2$ gets generators and relations for the `dual' Lie algebra $\pgl_2=\Lie(\PGL_2)$, which is evidently still not simple:  it has a $2$-dimensional abelian $p$-nilpotent ideal $\langle e,f\rangle$ on which $\ad(h)$ acts as the identity. However, setting $[h,e]=e$, $[h,f]=f$ and $[e,f]=h$ we do indeed get a simple Lie algebra $\s$ of dimension $3$, which is unique up to isomorphism.\footnote{Some prefer the cyclic basis $[x,y]=z$, $[y,z]=x$, $[z,x]=y$. Colloquially, the Lie algebra $\s$ is sometimes referred to as \emph{fake $\sl_2$}.}

Let us give a little more information about $\s$. On the one hand, we may identify $\s$ with the Cartan type Zassenhaus Lie algebra $W(1;2)'=\langle\del,x^{(1)}\del,x^{(2)}\del\rangle$ in characteristic $2$; see \cite{SF88} for a definition, using divided powers. On the other hand, one may see $\s$ as part of a larger family of simple orthogonal Lie algebras in characteristic $2$ analogous to those in type $B$. Following \cite[4.6.1]{BGL09} one sets $\hat\o_{2n+1}$ to be the Lie algebra defined by a Cartan matrix of type $B_n$, after modifying the commutator relations by halving the end column\footnote{We use the Bourbaki notation \cite{Bourb05}, so the long roots are on the left of the Dynkin diagram.}. The resulting commutator relations define a simple Lie algebra; indeed we have $\s\cong \hat \o_3$. It is a simple calculation to see that the algebra $\s$ is not restricted: it has codimension $2$ in its minimal $p$-envelope, $\s_p$, with basis $\{E,e,h,f,F\}$ where $[E,F]=h$, $[E,f]=e$, $[F,e]=f$ and $h$ centralises $E$ and $F$. We note that the irreducible representations of $\s$ were described in \cite{Dol78}.

The statements of our results are given in terms of the description of nilpotent elements in the Lie algebras of classical groups in terms of certain standard forms, developed by Hesselink \cite{MR0525621} and presented in \cite{LS12}. Moreover, our proofs rely on explicit representatives of these classes as sums of root vectors, utilising some earlier work \cite{KST21}. Since extracting the representative of a general class from \emph{op.~cit.} is not entirely straightforward, we suspect readers will appreciate the very explicit formulas for the representatives which we give in Section~\ref{sec:standardforms} below.  

Theorems \ref{thm:mainfsl2}, \ref{thm:mainpgl2} and \ref{thm:mainsl2} give a complete list of those nilpotent orbits $G(k)\cdot e$ such that $e$ lies in a subgroup isomorphic with $\s$, $\pgl_2$ and $\sl_2$ respectively. When $e\in\pgl_2$ or $e\in\s$, there is an element $h$ such that $[h,e]=e$. And when such an $h$ exists, the Lie-theoretic normaliser $\n_\g(\k\cdot e)$ is stricly larger than the centraliser $\c_\g(\k\cdot e)$; another way to say this is that the Lie-theoretic automiser $\n_\g(\k\cdot e)/\c_\g(e)$ is non-trivial (and must have dimension $1$). In Theorem \ref{thm:heeqe}, we give a complete list of the orbits where this happens, whose members turn out to be surprisingly sporadic. We then use that result to inform the proofs of Theorems \ref{thm:mainfsl2} and \ref{thm:mainpgl2}.

When $p > 2$, the automiser of $\k\cdot e$ is always $1$-dimensional: either $e$ lies in an $\sl_2$-subalgebra; or $p=3$, $G \cong G_2$ and $e$ has label $A_1^{(3)}$ where one can do a direct check. Hence the calculation of the automiser is only interesting when $p=2$, and Theorem \ref{thm:heeqe} then answers that line of enquiry. 

For other recent work in this direction, see \cite{GP21,PT23}.

\subsection*{Acknowledgments}
The first author is supported by a Leverhulme Trust Research Project Grant RPG-2021-080 and the second author is supported by an EPSRC grant EP/W000466/1. For the purpose of open access, the authors have applied a Creative Commons Attribution (CC BY) licence to any Author Accepted Manuscript version arising from this submission.

\section{Preliminaries and notation}
\subsection{A presentation for \texorpdfstring{$\g$}{g}}\label{sec:present}

Throughout, $\k$ will be an algebraically closed field of characteristic $2$ and $G$ will be a simple algebraic $\k$-group. Let $T$ be a maximal torus of $G$ of rank $n$, giving rise to a root datum $(X(T),\Phi,Y(T),\Phi^\vee)$ that determines $G$ up to isomorphism. Recall this means $X(T)=\Hom(T,\Gm)$ is the weight lattice, which is a free $\Z$-module of rank $n$ and $Y(T)$ is the $\Z$-linear dual space $\Hom(\Gm,T)$, which come with a perfect pairing $\langle\_,\_\rangle: X(T)\times Y(T)\to \Z$; the set of roots $\Phi\subseteq X(T)$ is the set of non-zero weights of $T$ on $\g=\Lie(G)$; and $\alpha\mapsto\alpha^\vee$ is a map from $\Phi$ into $Y(T)$ with image denoted $\Phi^\vee$, such that $s_\alpha(\beta)=\beta-\langle \beta,\alpha^\vee\rangle$---in particular, $\langle \alpha,\alpha^\vee\rangle=2$ and $\langle \alpha,\beta^\vee\rangle\in\Z$.

The Cartan matrix is $\CC  = (\langle \alpha_i ,\alpha_j^\vee \rangle)_{ij}$. A choice of basis $\{\omega_1,\dots,\omega_n\}$ for the lattice $X(T)$ is equivalent to a factorisation \begin{equation}C=\A\B^T,\label{eqn:fact}\end{equation} where $\A = (\langle \alpha_i, \omega_j^\vee \rangle)_{ij}$ and $\B = (\langle \omega_i, \alpha_j^\vee \rangle)_{ij}$. In Table \ref{t:AsandBs} we give choices for $\A$ and $\B$ in certain cases. 

\begin{table}\small\label{t:AsandBs}\begin{center}
  \setlength\extrarowheight{2pt}
  \begin{tabular}{c|c|c}
  Description of $G$ & $\A$ & $\B$\\
  \hline
  \parbox{1.5cm}{\vspace{4pt}simply\\connected\vspace{3pt}} & $\CC$ & $I_n$ \\
  
  adjoint & $I_n$ & $\CC$ \\
  
  &&\\
  
  $\SO_{2n}$ & $\begin{bmatrix}1 & 0 & 0 & \dots & 0 & 0& 0\\
    0 & 1 & 0 & \dots & 0 & 0& 0\\
    0 & 0 & 1 & \dots & 0 & 0& 0\\
    \vdots & \vdots & \vdots & \ddots & \vdots& \vdots & \vdots \\
    0 & 0 &  0& \dots & 1 & 0 & 0 \\
    0 & 0 &  0& \dots & 0 & 1 & 0 \\
    0 & 0 & 0 & \dots & 0 & 1 & 2\end{bmatrix}$
  
  & 
  
  $\begin{bmatrix}2 & -1 & 0 & \dots & 0 & 0& 0\\
    -1 & 2 & -1 & \dots  & 0 & 0& 0\\
    0 & -1 & 2 & \dots & 0 & 0& 0\\
    \vdots & \vdots & \vdots & \ddots & \vdots& \vdots & \vdots \\
    0 & 0 &  0& \dots & 2 & -1 & 0 \\
    0 & 0 &  0& \dots & -1 & 2 & -1 \\
    0 & 0 & 0 & \dots & -1 & 0 & 1\end{bmatrix}$\\
  
   &&\\
  
  \parbox{1.5cm}{$\HSpin_{2n}$\\ ($n$ even)} & $\begin{bmatrix}1 & 0 & 0 & 0 &  \dots & 0&0&0&0\\
      0 & 1 & 0 & 0 & \dots & 0 & 0 & 0 & 0\\
      0 & 0 & 1 &0 & \dots &  0 & 0 & 0 & 0\\
      0 & 0 & 0 &1 & \dots & 0 & 0 & 0 & 0\\
      \vdots & \vdots & \vdots &\vdots & \ddots & \vdots& \vdots& \vdots& \vdots\\
      0 & 0 & 0 & 0 & \dots  & 1 & 0 & 0 &0 \\
      0 & 0 & 0 & 0 & \dots  & 0 & 1 & 0 &0 \\
      1 & 0 & 1 & 0 & \dots  & 1 & 0 & 2 &0 \\
      0 & 0 & 0 & 0 & \dots  & 0 & 0 & 0 & 1\end{bmatrix}$
  
  & 
  
  $\begin{bmatrix}
      2 & -1 & 0 & 0 &  \dots & 0 & 0 &  -1 & 0\\
      -1 & 2 & -1 & 0 & \dots & 0 & 0 & 1 & 0\\
      0 & -1 & 2 & -1 & \dots &  0 & 0 & -1 & 0\\
      0 & 0 & -1 & 2 & \dots & 0 & 0 & 1 & 0\\
      \vdots & \vdots & \vdots &\vdots & \ddots & \vdots& \vdots& \vdots& \vdots\\
      0 & 0 & 0 & 0 & \dots  & 2 & -1 & -1 & 0 \\
      0 & 0 & 0 & 0 & \dots  & -1 & 2 & 0 & -1 \\
      0 & 0 & 0 & 0 & \dots  & 0 & -1 & 1 & 0 \\
      0 & 0 & 0 & 0 & \dots  & 0 & -1 & 0 & 2\end{bmatrix}$    \\
  
  \end{tabular}
  \end{center}\caption{Choices of factorisations $\CC=\A\B^T$ of the Cartan matrix. We take $\A$ as in  \cite[4.4.4]{Roozemond} except in for $\HSpin_{2n}$, where we swapped the roles of $\alpha_{n-1}$ and $\alpha_n$ for later convenience.}\end{table}

\label{sec:formulae}
Some explicit calculation with root vectors in $\g$ are unavoidable and so we fix Chevalley bases following \cite{MR2531218}. This means giving a presentation of an integral form $\g_\Z$ of $\g$ as follows. Let $\h_\Z:=\langle \sh_1,\dots,\sh_n\rangle$ be a free $\Z$-module of rank $n$ and take the direct sum of $\h_\Z$ with a free $\Z$-module of rank $|\Phi|$ to form 
\[\g_\Z := \Z \sh_1\oplus\dots\oplus\Z \sh_n\oplus \bigoplus_{\alpha\in \Phi}\Z \e_\alpha.\] We then make $\g_\Z$ into a Lie algebra by giving its structure constants. For all $1\leq i,j\leq n$ and $\alpha,\beta\in \Phi$, let: 
\begin{align}&[\sh_i,\sh_j]=0;\tag{CB1}\\
&[\sh_i,\e_\alpha]=\langle\alpha,\omega_i^\vee \rangle;\tag{CB2}\label{cb2}\\
&[\e_{\alpha},\e_{-\alpha}]=\sum_{i=1}^n\langle \omega_i,\alpha^\vee\rangle \sh_i;\tag{CB3}\label{cb3}\\
&[\e_\alpha,\e_\beta]:=\begin{cases}N_{\alpha,\beta} \e_{\alpha+\beta} &\text{ if }\alpha+\beta\in\Phi, \\0 &\text{otherwise.}\end{cases}\tag{CB4}
\end{align}
The integer $N_{\alpha+\beta}$ is defined to be $\epsilon_{\alpha\beta}p_{\alpha,\beta}+1$ where $p_{\alpha,\beta}$ is the largest integer such that $\beta - p_{\beta,\gamma} \gamma \in \Phi$. (The coefficient $\epsilon_{\alpha\beta}$ is always $\pm 1$, so in $\k$ it may be ignored.)

It is a theorem of Chevalley \cite[Thm.~2]{MR2531218} that (CB1)--(CB4) determine a presentation of $\g$. More precisely, there is an isomorphism of Lie algebras $\g_Z\otimes_\Z \k\cong \g$; and under that isomorphism the abelian subalgebra $\h_\Z\otimes \k$ is sent to the Cartan subalgebra $\h$, and $\Z\e_\alpha\otimes_\Z \k$ goes to $\k\cdot\e_\alpha$. 

The equations above imply that $\g$ has a grading by root height. Given a root $\beta = \sum_{i=1}^n a_i \alpha_i \in \Phi$, the height of $\gamma$ is $\text{ht}(\gamma) = \sum_{i=1}^n a_i$. The roots of positive height are the positive roots $\Phi^+$ and those of negative height are the negative roots $\Phi^-$. Define also $\tilde{\alpha}$ to be the highest root in $\Phi$. We have:
\begin{align*}\g&=\bigoplus_{i\in\Z}\g_i \quad \text{ is a grading of $\g$, where }\quad \g_0=\h\quad \text{ and }\quad \g_i=\bigoplus_{\mathrm{ht}(\alpha)=i}\k e_\alpha\end{align*} 
We also write $\g_{\geq i}:=\bigoplus_{\text{ht}(\alpha)\geq i}\g_\alpha$ with $\g_{\leq i}$ defined similarly. Put $\g_+:=\g_{\geq 0}$, $\g_-:=\g_{\leq 0}$. And lastly set $\pi_i(x)$ be the projection of $x$ to $\g_i$.

We record some formulas that follow from the choices for $\A$ in Table \ref{t:AsandBs}. For brevity we write \[e_i:=e_{\alpha_i}\qquad \text{ and }\qquad f_i:=e_{-\alpha_i} \qquad \text{ for }\alpha_i \in \Delta.\]

When $G$ is adjoint, we have $\A=I$ and so if $h:=\sum_{i=1}^n h_i$, we get
\begin{equation}
  [h,e_{\gamma}]=\mathrm{ht}(\gamma).\label{eqn:headj}
\end{equation}

When $G=\SO_{2n}$, we have $1 \leq i \leq n$ and all roots $\gamma = \sum_j c_j \alpha_j$
\begin{align} [h_{i}, e_{\gamma}] = \left(\left(\sum_{j=1}^{n-2} \delta_{i,j} c_j\right) + \delta_{i,n-1} (c_{n-1} + c_{n}) \right) e_\gamma. \label{eq:heSO2n} \end{align}

When $G=\HSpin_{2n}$, for $1 \leq i \leq n-2$ set $\epsilon(i) = i \pmod 2$ and set $\epsilon(n-1) = \epsilon(n) = 0$. Then for $1 \leq i \leq n$ and all roots $\gamma = \sum_j c_j \alpha_j$ we have

\begin{align} [h_{i}, e_{\gamma}] = \left(\left(\sum_{j=1}^{n-2} \delta_{i,j} c_j\right) + \delta_{i,n-1} c_{n-1} + \epsilon(i) c_{n-1}\right) e_\gamma. \label{eq:heHSpin2n} \end{align}

\begin{remark}
Our choice of matrix $\A$ when $G\cong\HSpin_{2n}$ slightly obscures the useful fact that if one restricts to the Levi subsystem corresponding to the nodes $1, \ldots, \alpha_{n-2}, \alpha_n$, then the corresponding $n-1\times n-1$ minor of $\A$ is  an identity matrix. 
\end{remark}

\subsection{Nilpotent orbits} \label{sec:nilporbits}
Since char $\k>0$, the Lie algebra $\g$ comes with a $[p]$ map, $x\mapsto x^{[p]}$, which is $p$-semilinear in $\k$ and satisfies $\ad(x^{[p]})=(\ad(x))^p$.
An element $x\in\g$ is \emph{nilpotent} if $x^{[p]^r}=0$ for some $r\geq 0$; it is \emph{semisimple} if $x$ is in the $\k$-span of $x^{[p]},\dots,x^{[p]^r}$ for some $r\geq 0$; 
it is \emph{toral} if $x^{[p]}=x$; and a subalgebra of $\g$ with a basis of toral elements is called a \emph{torus}. Since $\g=\Lie(G)$ for $G$ simple, its Cartan subalgebras are maximal tori. The properties above are invariant under conjugacy by $g\in G(k)$. In particular, an orbit $\OO_e=G(k)\cdot e$ of elements of $\g$ is called a nilpotent orbit if it consists of nilpotent elements. The union of all nilpotent orbits of $\g$ is denoted $\N:=\N(\g)$ and called the \emph{nilpotent cone}. 

For the general theory of nilpotent orbits and algebro-geometric properties of $\N$ we refer the reader to \cite{Jan04}. For complete data of the classification of the nilpotent orbits we recommend \cite{LS04}, but we give a digest of the salient points below.

\subsection{Standard forms and representatives for nilpotent classes for classical groups} \label{sec:standardforms}
Let $G$ be classical. Suppose that $\phi: G \rightarrow H$ is a central isogeny.
Then $d\phi$ induces a bijection between the nilpotent cones of $\Lie(G)$ and $\Lie(H)$; for details, see \cite[Prop.~2.7]{Jan04}. Therefore when enumerating nilpotent orbits it is sufficient for us to assume $G$ is one of $\SL(V)$, $\Sp(V)$ or $\SO(V)$ and we refer to $V$ as the natural module for $G$.  

A \emph{standard form} for a nilpotent element $e$ of $\g = \Lie(G)$ is given in \cite{LS12}. It consists of listing indecomposable summands occurring in the restriction of $V$ to $\k[e]$ predicated on the bilinear or quadratic form that $G$ preserves on $V$; ultimately the possibilities are parametrised by a small number of decreasing sequences of positive integers $\SF(e)$ that we call the \emph{specification} of $e$. It is a theorem that the orbit of $e$ is determined by $\SF(e)$, with one collection of exceptions---see Remark \ref{rem:twoclassesD}. The possible standard forms and their specifications are described in Table \ref{tab:stdforms}. 

In \cite{KST21}, the authors give a recipe for writing down a representative of $e$ in terms of root vectors given the specification $\SF(e)$. Table \ref{tab:stdreps} works out this recipe explicitly in all cases. The main tool for generating representatives is the Borel--de Siebenthal algorithm, which is a well-known procedure for finding bases for root subsystems; here a \emph{root subsystem} $\Psi$ is a subset of $\Phi$ which is symmetric in the sense that $\Psi=-\Psi$ and so that $g_\Psi:=\langle e_{\psi},h_i\mid \psi\in\Psi,1\leq i \leq n\rangle$ is a Lie subalgebra with root system $\Psi$. In Table \ref{tab:stdreps} we give for each $e$ a natural subsystem $\Psi:=\Psi(e)$ such that $e\in g_\Psi$. The subsystem $\Psi(e)$ has a base $\Delta_\Psi$ which is very nearly the set of roots $J:=J_1\cup J_2$ on which $e$ is supported---the table explains how to modify $J_2$ to $J_2'$ so that $\Delta_\Psi=J_1\cup J_2'$.

For the interested reader, we explain the  modification of $J_2$ to $J_2'$ in a representative example where $\Phi$ has type $B_n$. If $\SF(e)=( (k_i)_{i=1}^r, (l_i)_{i=1}^s, (m_i)_{i=1}^s, m)$, then 
\[\Psi=A_{k_1-1}\dots A_{k_r-1}D_{m_1}\dots D_{m_s}B_{m-1}.\] Applying the Borel--de Siebenthal algorithm gives a base for the $D_{m_i}$ factor as 
\[\alpha_{v_i+m_i-1}, \ldots, \alpha_{v_i+1}, -\tilde{\alpha}_i = -(\alpha_{v_i+1} + 2\alpha_{v_i+1} + \cdots + 2\alpha_{n}),\] where $\tilde{\alpha}_i$ is the highest root of the type $B$ subsystem with base $\{\alpha_{v_i+1},\dots,\alpha_n\}$, and the integer $v_i$ is defined in Table \ref{tab:stdreps}. A straightforward Weyl group calculation shows that \[ \alpha_{v_i+1}, \ldots, \alpha_{v_i+m_i - 1}, \ \ \widetilde{\beta_i} = \alpha_{v_{i+1}-1} + 2 \alpha_{v_{i+1}} + \cdots + 2 \alpha_{n-1} + 2\alpha_{n}\] is also a base for the given $D_{m_i}$ subsystem, now consisting only of positive roots; this is what we use as $\Delta_{\Psi}$.

\begin{landscape}
\begingroup
\renewcommand*{\arraystretch}{1.5}
\begin{table}
\begin{tabular}{llll}
\hline
$G$ & $V \downarrow \k[e]$ & Specification $\SF(e)$ & Conditions \\
\hline\hline
$A_n$ & $\sum^{r}_{i=1} V(k_i)$ & $( (k_i)_{i=1}^r)$ & $\sum^{r}_{i=1} k_i = n+1$; \\
& & & $(k_i)_{i=1}^{r}$ decreasing \\
\hline

$B_n$ & $\sum^{r}_{i=1} W(k_i) + \sum^{s}_{i=1} W_{l_i}(m_i)$ & $( (k_i)_{i=1}^r, (l_i)_{i=1}^s, (m_i)_{i=1}^s, m)$ & $\sum^{r}_{i=1} 2k_i + \sum^{s}_{i=1} 2m_i + 2m-1 = 2n+1$; \\
& $+D(m)$ & & $(k_i)_{i=1}^{r}$ decreasing; \\
& & & $(m_i)_{i=1}^{s}$, $(l_i)_{i=1}^{s}$, $(m_i-l_i)_{i=1}^{s}$ strictly decreasing; \\
& & & $\frac{m_i + 1}{2} < l_i \leq m_i$ and $m < l_s$ \\
\hline

$C_n$ & $\sum^{r}_{i=1} W(k_i) + \sum^{s}_{i=1} W_{l_i}(m_i)$  & $( (k_i)_{i=1}^r, (l_i)_{i=1}^s, (m_i)_{i=1}^s, (n_i)_{i=1}^t)$ & $\sum^{r}_{i=1} 2k_i + \sum^{s}_{i=1} 2m_i + \sum^{t}_{i=1} 2n_i = 2n$; \\
&$+ \sum^{t}_{i=1} V(2n_i)$ & & $(k_i)_{i=1}^{r}$, $(n_i)_{i=1}^{t}$  decreasing; \\
& & & $(m_i)_{i=1}^{s}$, $(l_i)_{i=1}^{s}$, $(m_i-l_i)_{i=1}^{s}$ strictly decreasing; \\
& & & multiplicity of any $n_i$ is at most two; \\
& & & $0 < l_i < \frac{m_i}{2}$; \\
& & & for all $i,j$ either $n_j > m_i - l_i$ or $n_j < l_i$ \\
\hline

$D_n$ & $\sum^{r}_{i=1} W(k_i) + \sum^{s}_{i=1} W_{l_i}(m_i)$ & $( (k_i)_{i=1}^r, (l_i)_{i=1}^s, (m_i)_{i=1}^s)$ & $\sum^{r}_{i=1} 2k_i + \sum^{s}_{i=1} 2m_i = 2n$; \\
& & & $(k_i)_{i=1}^{r}$ decreasing; \\
& & & $(m_i)_{i=1}^{s}$, $(l_i)_{i=1}^{s}$, $(m_i-l_i)_{i=1}^{s}$ strictly decreasing; \\
& & & $\frac{m_i + 1}{2} < l_i \leq m_i$ \\
\hline
\end{tabular}
\vspace{10pt}
\caption{Standard forms and specifications of nilpotent orbits in classical type} \label{tab:stdforms}
\end{table}

\begin{table}
\begin{tabular}{lllll|l}
\hline
$G$ & $J_1$ & $J_2$  & $J'_2$ &  Overalgebra & Root definitions \\
\hline\hline

$A_n$ & $\bigcup\limits_{i=1}^r \{ \alpha_{u_i+j} \}_{j = 1}^{k_i-1}$ &  & & $A_{k_1 -1} \ldots A_{k_r -1}$ & \\

\hline
$B_n$ & $\bigcup\limits_{i=1}^r \{ \alpha_{u_i+j} \}_{j = 1}^{k_i-1}$ &  & & $A_{k_1 -1} \ldots A_{k_r -1}$ & $\beta_i =  \alpha_{v_{i+1} - 2(m_i - l_i)-1} + \cdots + \alpha_{v_{i+1}-1} + 2\alpha_{v_{i+1}} + \cdots + 2\alpha_{n-1} + 2\alpha_{n}$ \\

& $\bigcup\limits_{i=1}^s \{ \alpha_{v_i+j}\}_{j = 1}^{m_i-1}$ & $\{\beta_i\}_{i = 1}^{s}$ &  $\{\widetilde{\beta_i}\}_{i = 1}^{s}$  & $D_{m_1} \ldots D_{m_s}$ & $\widetilde{\beta_i} = \alpha_{v_{i+1}-1} + 2\alpha_{v_{i+1}} + \cdots + 2\alpha_{n-1} + 2\alpha_{n}$ \\

& $\{ \alpha_{j}\}_{j = n-m+2}^{n}$ & &  & $B_{m-1}$ &   \\

\hline
$C_n$ & $\bigcup\limits_{i=1}^r \{ \alpha_{u_i+j} \}_{j = 1}^{k_i-1}$ & & & $A_{k_1 -1} \ldots A_{k_r-1} $ & $\beta_i = 2\alpha_{v_{i}+l_i} + \cdots + 2\alpha_{n-1} + \alpha_n$ \\

& $\bigcup\limits_{i=1}^s \{ \alpha_{v_i+j} \}_{j = 1}^{m_i-1}$ & $\{\beta_i\}_{i = 1}^{s}$ & $\{\widetilde{\beta_i}\}_{i = 1}^{s}$ & $C_{m_1} \ldots C_{m_s}$ & $\widetilde{\beta_i} = 2\alpha_{v_{i+1}} + \cdots + 2 \alpha_{n-1} + \alpha_n$ \\

& $\bigcup\limits_{i=1}^t \{ \alpha_{w_i+j} \}_{j = 1}^{n_i-1}$ & $\{\gamma_i\}_{i = 1}^{t}$ & $\{\gamma_i\}_{i = 1}^{t}$ & $C_{n_1} \ldots C_{n_t}$ & $\gamma_i = \begin{cases} 
2\alpha_{w_i + n_i} + 2 \alpha_{w_i + n_i+1} + \cdots + 2 \alpha_{n-1} + \alpha_{n} & \ \text{if } w_i + n_i < n, \\
\alpha_{n} & \ \text{if } w_i + n_i = n.\\
\end{cases}$  \\
\hline
$D_n$ & $\bigcup\limits_{i=1}^r \{ \alpha_{u_i+j} \}_{j = 1}^{k_i-1}$ &  &  & $A_{k_1 -1} \ldots A_{k_r -1}$ & $\beta_i = \begin{cases}  \parbox{0.45\textwidth}{$\alpha_{v_{i+1} - 2(m_i - l_i)-1} + \cdots + \alpha_{v_{i+1}-1}$\\\phantom{m} $+ 2\alpha_{v_{i+1}} + \cdots + 2\alpha_{n-2} + \alpha_{n-1} + \alpha_{n}$} & \ \text{if } i < s, \\
\alpha_{v_{i+1} - 2(m_i - l_i)} + \cdots + \alpha_{n} & \ \text{if } i = s. \\ \end{cases}$ \\

& $\bigcup\limits_{i=1}^s \{ \alpha_{v_i+j}\}_{j = 1}^{m_i-1}$ & $\{\beta_i\}_{i = 1}^{s}$ & $\{\widetilde{\beta_i}\}_{i = 1}^{s}$ & $D_{m_1} \ldots D_{m_s}$ & $\widetilde{\beta_i} = \begin{cases} \alpha_{v_{i+1}-1} + 2\alpha_{v_{i+1}} + \cdots + 2\alpha_{n-2} + \alpha_{n-1} + \alpha_{n} & \ \text{ if } i < s, \\
\alpha_{n} & \ \text{ if } i = s. \\ \end{cases}$ \\
\hline
\end{tabular}
\vspace{5pt}
\begin{align*}
\textit{where: }\qquad\qquad& u_i = \sum_{j=1}^{i-1} k_j & \text{for } i =1, \ldots, r+1,&\qquad\qquad\qquad\qquad \\
  & v_i = u_{r+1} + \sum_{j=1}^{i-1} m_j & \text{for } i=1, \ldots, s+1, \\
  & w_i = v_{s+1} + \sum_{j=1}^{i-1} n_j & \text{for } i=1, \ldots, t+1.
  \end{align*}
  \flushleft\emph{In each case, the representative for the nilpotent orbit $\OO_e$ with specification $\SF(e)$ is $e = \sum_{j \in J} e_j$, where $J = J_1 \cup J_2$. The element $e$ is constructed inside an overalgebra containing $e$ which has base $J_1 \cup J'_2$.}
\caption{Standard representatives of nilpotent orbits in classical types \label{tab:stdreps}}

\end{table}
\endgroup
\end{landscape}

\begin{remark}The Dynkin diagrams below should be of assistance in seeing the connection between representatives and the specification $\SF(e)$, as well as the natural overalgebras.

\underline{Type $A$}

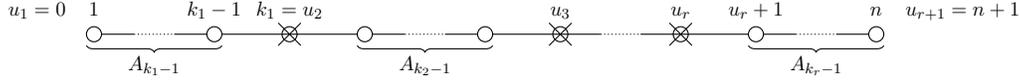
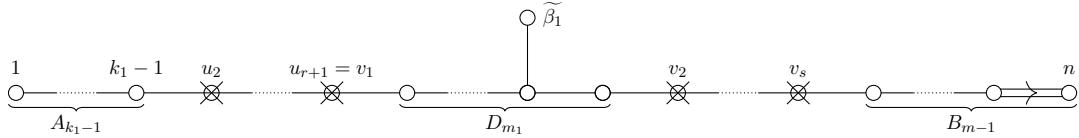
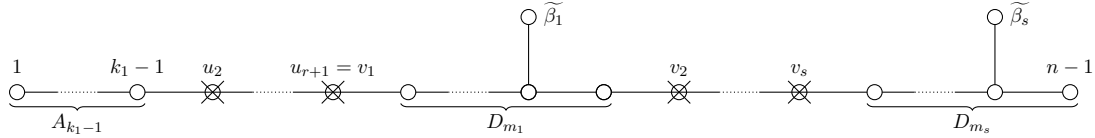
\begin{center}
\newcommand\circleRoot[1]{
\draw (root #1) circle (0.1cm);}
\dynkin[labels*={u_1 = 0}, */.style={fill=white, draw=white}]A1
\begin{dynkinDiagram}[labels*={1,k_1-1,k_1 = u_2,,,u_3,u_r,u_r+1,n}, x/.style={root radius = 0.15cm}]A{o.oxo.ox.xo.o}
\foreach\r in {3,6,7} {\circleRoot \r}
\dynkinBrace[A_{k_1-1}]12
\dynkinBrace[A_{k_2-1}]45
\dynkinBrace[A_{k_r-1}]{8}{9}
\end{dynkinDiagram}
\dynkin[labels*={u_{r+1} = n+1}, */.style={fill=white, draw=white}]A1
\end{center}

\underline{Type $B$}

\begin{center}
\makeatletter
\newcommand{\extraNode}[6]%
{%
\dynkinPlaceRootRelativeTo{#1}{#2}{#3}{#4}{#5}
\dynkinDefiniteSingleEdge{#1}{#2}
\dynkinRootMark{o}{#1}
\advance\dynkin@nodes by 1
\dynkinLabelRoot{#1}{#6}
}%
\newcommand\circleRoot[1]{
\draw (root #1) circle (0.1cm);}
\begin{dynkinDiagram}[labels*={1,k_1-1,u_2,u_{r+1} = v_1,,,,v_2,v_{s},,,n}, x/.style={root radius = 0.15cm}]B{o.ox.xo.oox.xo.oo}
\foreach\r in {3,4,6,7,8,9} {\circleRoot \r}
\dynkinBrace[A_{k_1-1}]12
\dynkinBrace[D_{m_1}]57
\dynkinBrace[B_{m-1}]{10}{12}
\extraNode{13}{6}{north}{right}{left}{\widetilde{\beta_1}}
\end{dynkinDiagram}

\end{center}

\underline{Type $D$}

\begin{center}
\makeatletter
\newcommand{\extraNode}[6]%
{%
\dynkinPlaceRootRelativeTo{#1}{#2}{#3}{#4}{#5}
\dynkinDefiniteSingleEdge{#1}{#2}
\dynkinRootMark{o}{#1}
\advance\dynkin@nodes by 1
\dynkinLabelRoot{#1}{#6}
}%
\newcommand\circleRoot[1]{
\draw (root #1) circle (0.1cm);}
\begin{dynkinDiagram}[labels*={1,k_1-1,u_2,u_{r+1} = v_1,,,,v_2,v_{s},,,n-1}, x/.style={root radius = 0.15cm}]A{o.ox.xo.oox.xo.oo}
\foreach\r in {3,4,6,7,8,9} {\circleRoot \r}
\dynkinBrace[A_{k_1-1}]12
\dynkinBrace[D_{m_1}]57
\dynkinBrace[D_{m_s}]{10}{12}
\extraNode{13}{6}{north}{right}{left}{\widetilde{\beta_1}}
\extraNode{14}{11}{north}{right}{left}{\widetilde{\beta_s}}
\end{dynkinDiagram}
\end{center}\end{remark}

\begin{remark} \label{rem:twoclassesD}
  For type $D_n$, the specification $\SF(e) = ( (k_i)_{i=1}^r, (l_i)_{i=1}^s, (m_i)_{i=1}^s)$ determines the $\O(V)$-orbit of the nilpotent element $e \in \Lie(\SO(V))$. By \cite[Prop.~5.25]{LS12}, the class of $e$ splits into two $\SO(V)$-orbits if and only if $s=0$ and all $k_i$ are even. In that case, one sees that the representative given in Table \ref{tab:stdreps} is contained in the standard Levi subalgebra of type $A_{n-1}$ with base $1, \ldots, n-1$. For a representative in the other $\SO(V)$-class, use $J'$ instead of $J$, where $J'$ is formed from $J$ by replacing every occurrence of $\alpha_{n-1}$ with $\alpha_n$. 
\end{remark}

\begin{example}[Regular elements] If $e$ is a regular element, say $e=\sum_{i=1}^n e_i$, then in types $A_n$, $B_n$, $C_n$ and $D_n$ respectively, we have \[\SF(e)=((n+1)),\quad ((),(),(),n+1),\quad ((),(),(),(n))\quad\text{ and }\quad ((),(n),(n)).\] 
\end{example}

Finally for this subsection, we take the opportunity to record a technical lemma regarding the special roots defined in Table \ref{tab:stdreps}. The proof is essentially by inspection and left to the reader.

\begin{lemma}\label{rootfacts}Suppose $G$ has type $B$ or $D$. Then:
\begin{enumerate}
  \item the difference $\beta_i - \beta_j \not \in \Phi$ for  $i,j \in \{1, \ldots, s\}$;
  \item there exists $\gamma \in \Phi^+$ with $\beta_j - \gamma = \alpha_i$ if and only if $l_j < m_j$ and $i = v_{j+1} - 2(m_j-l_j)$.  
\end{enumerate}
\end{lemma}

\subsection{Standard forms and representatives for nilpotent classes for exceptional groups} When $G$ is exceptional, there is no natural module available from which to derive a classification of nilpotent orbits. Instead the Bala--Carter system of labelling gives root systems of Levi subalgebras in which the element is distinguished. The reader is invited to consult \cite[\S9]{LS04} for an overview. For the purposes of this paper, we need only the finite list of orbit representatives from \cite{SteMin}, which are implemented in Magma (and GAP) and available at \url{https://github.com/davistem/nilpotent_orbits_magma}.

\subsection{Identification of orbits under embeddings}
Given an embedding $\varphi:G\subseteq G'$ of simple algebraic groups, one gets a $G$-equivariant embedding $\dd\varphi:\N\to \N'$ of nilpotent cones. Typically, several $G$-orbits in $\N$ fuse into a single $G'$-orbit after applying $\dd\varphi$. We spell out what happens for two particular instances of $\varphi$.

\label{BtoD}
\underline{$\varphi:\O_{2n+1}\to \O_{2n+2}$.}\vspace{2pt} Let $V$ be a non-degenerate codimension-$1$ subspace of a $(2n+2)$-dimensional space $V_0$ on which a non-singular quadratic form has been fixed. Then there is a natural embedding $\varphi:\O(V)\to\O(V_0)$. This embedding is used in \cite[Section~5.6]{LS12} to describe the nilpotent cone of $\O(V)$. Let $e \in \N(\O(V))$ with specification  
$\SF(e) = ( (k_i)_{i=1}^r, (l_i)_{i=1}^s, (m_i)_{i=1}^s, m)$ and suppose $e' = \dd\varphi(e) \in \N'$. Define $m_{s+1} = m$. Then $\SF(e') = ( (k_i)_{i=1}^r, (l_i)_{i=1}^s, (m_i)_{i=1}^{s+1})$. 

\underline{$\varphi:\O_{2m}\to \Sp_{2m}$.}\vspace{2pt}  \label{OtoSp}A non-degenerate quadratic form on $V$ preserved by $\O(V)$ gives rise to a bilinear form on $V$, which since $p=2$, is alternating. Thus we get an inclusion $\varphi:\O(V)\subseteq \Sp(V)$. Indeed, an easy calculation shows that $\Lie(\O(V))$ identifies with the derived subalgebra of $\Lie(\Sp(V))$. Each orbit $\OO\subseteq \N$ is characterised by admitting a representative $e\in\Lie(O(V))$ with \[V \downarrow \k[e] = \sum^{r}_{i=1} W(k_i) + \sum^{s}_{i=1} W_{l_i}(m_i).\] In \cite[5.1(B)]{LS12} we learn that $W_l(m)$ is a non-degenerate space of dimension $2m$ which is a sum of two totally singular spaces of dimension $m$. These two spaces are then also totally isotropic  for the associated bilinear form and so $e$ has the same class over $\Sp(V)$ as a nilpotent element $e'$, where 
  \[V \downarrow \k[e'] = \sum^{r}_{i=1} W(k_i) + \sum^{s}_{i=1} W(m_i).\] 
(The reader should note that the two $W_l(m)$'s defined for $\O(V)$ and $\Sp(V)$ in \cite[\S5.1]{LS12} have no relationship with each other.) Also, observe that for any admissible sequence of $k_i$ such that there is $e'\in \Lie(\Sp(V))$ with $V\downarrow \k[e]= \sum_{i=1}^r W(k_i)$, then there is $e\in \Lie(\O(V))$ with $\dd\varphi(e)=e'$. Hence,
\begin{equation}\label{eq_inso}\text{if }e\in \sp_{2n}\text{ with }\SF(e)=((k_i)_{i=1}^r,(l_i)_{i=1}^s,(m_i)_{i=1}^s,(n_i)_{i=1}^t)\text{ then }e\in \so_{2n}\Leftrightarrow s=t=0.\end{equation}

\section{Solving \texorpdfstring{$[h,e]=e$}{[h,e]=e}} \label{sec:heeqe}
In this section $G$ is a simple $\k$-group, classical or exceptional, and $e$ is a nilpotent element of $\g$. In pursuit of a classification of the orbits $\OO_e$ for which $e\in\s\subseteq \g$ or $e\in\pgl_2\subseteq \g$, we search first for those $\OO_e$ for which there is an element $h$ satisfying $[h,e]=e$. In case $G$ is classical, recall the notation for the specifications $\SF(e)$ from Table \ref{tab:stdforms} and the representatives from Table \ref{tab:stdreps}.

\begin{theorem}\label{thm:heeqe}
There exists $h \in \g$ such that $[h,e]=e$ if and only if at least one of the following holds:
\begin{enumerate}
\item $G$ is adjoint, or more generally $[X(T): \Z\Phi]$ is coprime to $2$; 
\item $G\cong\SO_{2n}$; 
\item $G$ is of type $A_n$ and either $n \not \equiv 1 \pmod 4$ or $\SF(e)$ has at least one odd $k_i$;
\item $G\cong\Spin_{2n+1}$ and either $n \not \equiv 1, 2 \pmod 4$ or $\SF(e)$ has at least one odd $k_i$ or $m_i$;
\item \label{Cpartpglthm} $G\cong \Sp_{2n}$ and $\SF(e)$ has $s=t=0$;
\item $G\cong \Spin_{2n}$ and either $n \not \equiv 2 \pmod 4$ or $\SF(e)$ has at least one odd $k_i$ or $m_i$;
\item $G\cong\HSpin_{4n}$, and 
\begin{enumerate} 
\item $\SF(e)$ has at least one odd $k_i$ or $m_i$;
\item $\SF(e)$ has $s=0$ and the standard representative of $\OO_e$ is supported on $e_n$;
\end{enumerate}
\item $G$ is of type $E_6$; 
\item $G$ is of type $E_7$ and $e$ has one of the following labels: $E_{6}$, $E_{6}(a_{1})$, $A_{6}$, $D_{5}$, $E_{6}(a_{3})$, $(A_{5})'$, $A_{4}+A_{2}$, $D_{5}(a_{1})$, $A_{4}+A_{1}$, $A_{4}$, $A_{3}+A_{2}$, $D_{4}$, $D_{4}(a_{1})$, $(A_{3}+A_{1})'$, $2A_{2}+A_{1}$, $2A_{2}$, $A_{3}$, $A_{2}+2A_{1}$, $A_{2}+A_{1}$, $A_{2}$, $(3A_{1})'$, $2A_{1}$, $A_{1}$.  
\end{enumerate}

\end{theorem}

\begin{remark} The condition on $\SF(e)$ in part (\ref{Cpartpglthm}) can be restated as saying that $e$ is contained in a Levi subalgebra of type $A_{n-1}$. \end{remark}

The rest of this section is concerned with the proof of Theorem \ref{thm:heeqe}. The machine computation in the exceptional case will be discussed in Section \ref{sec:excep}. Therefore we work under the following assumptions until further notice:

\begin{center}\framebox{$G$ is classical and $e$ is a standard representative of $\g$.}\end{center}  

\begin{lemma} \label{lem:heeclassadj}
Theorem \ref{thm:heeqe} holds when $[X(T):\Z\Phi]$ is coprime to $2$ or $G\cong\SO_{2n}$. 
\end{lemma}

\begin{proof} 
If $[X(T):\Z\Phi]$ is coprime to $2$, the isogeny from $G$ to $G / Z(G)$ is separable and so we may reduce to the adjoint case. Thus the matrix $\A$ appearing in the factorisation $\CC=\A\B^T$ is the identity.

Let $h = \sum_{i=1}^n h_i$. Then since in all cases $e$ is a sum of root vectors, we have $[h,e] = e$ by Equation (\ref{cb2}). (For the case $G\cong\SO_{2n}$ we refer to (\ref{eq:heSO2n}) in Section \ref{sec:formulae}.)
\end{proof}

\begin{lemma} \label{l:heeqeC}
Suppose $G\cong\Sp_{2n}$. Then there exists $h\in\g$ with $[h,e]=e$ if and only if $e\in [\g,\g]$. Hence Theorem \ref{thm:heeqe} holds for $G$. 
\end{lemma}
\begin{proof}
If $[h,e]=e$, then clearly $e \in [\g,\g]$. On the other hand, $\m:=[\g,\g]=\Lie(M)$ for a closed subgroup $M \cong \SO_{2n}$ and so if $e\in \m$ then Lemma \ref{lem:heeclassadj} gives the result. The characterisation of $e$ in Theorem \ref{thm:heeqe} follows from (\ref{eq_inso}).
\end{proof}

Lemmas \ref{lem:heeclassadj} and \ref{l:heeqeC} deal completely with type $C$. So we may assume from now on that
\begin{center}\framebox{$G$ has type $A$, $B$ or $D$.}\end{center}

\begin{lemma}\label{l:hintorusgeneric} Suppose there exists $h\in\g$ such that $[h,e]=e$. Then there is some $h_0 \in \h$ such that $[h_0,e]=\pi_1(e)$. 
\end{lemma}
\begin{proof}Write $h = h_- + h_0 + h_+$ where $h_- \in \g_-$, $h_+ \in \g_+$ and $h_0 \in \g_0 = \h$. (Recall the notation from Section~\ref{sec:present}.)

Suppose $G$ is of type $A_n$. Since $e$ is a standard representative, it is a sum of simple root vectors and so $e \in \g_1$. Then $[h_-,e] \subseteq \g_{-}$, $[h_+,e] \subseteq \g_{\geq 2}$ and $[h_0,e] \subseteq \g_{1}$. Therefore, $[h,e]=[h_0,e]=e$. 

Now let $G$ be of type $B_n$ and suppose the statement of the theorem does not hold. Then $[h_0,e]\neq \pi_1(e)$, which implies there is $\alpha_i\in J_1$ with $[h_0,e_i]\neq e_i$.

Since $\pi_1([h_+,e])=0$, we must have $[h_-,e]$ supported on $e_i$. Recall that $e$ is a sum of root vectors for roots of height $1$ or roots of the form $\beta_j \in J_2$ (Table \ref{tab:stdreps}). As $\ad(h_-)$ takes $\g_1$ to $\g_{-}$, there must be $\beta_j\in J_2$ with $[h_-,e_{\beta_j}]$ supported on $e_i$. So $h_-$ is supported on $e_{-\gamma_1}$ for some $\gamma_1 \in \Phi^+$ with $\beta_j - \gamma_1 = \alpha_i$. Lemma \ref{rootfacts}(ii) shows that $m_j \neq l_j$ and $i = v_{j+1} - 2(m_j-l_j)$ and hence
\begin{align*} \gamma_1 = & (\beta_j - \alpha_{v_{j+1} - 2(m_j-l_j)})  \\
= & -\left( \alpha_{v_{j+1} - 2(m_j - l_j) + 1} +  \cdots + \alpha_{v_{j+1}} + 2\alpha_{v_{j+1}+1} + \cdots + 2\alpha_{n-2} + 2\alpha_{n-1} + 2\alpha_{n} \right). \end{align*}

We know $[h,e] = e$ and that $[h_0+h_+,e]$ is only supported on positive roots. Thus $[h_-,e]$ cannot be supported on negative roots. However, observe that $e$ is supported on $e_{v_{j+1} - 2(m_j-l_j) +1}$ and $h_-$ is supported on $e_{-\gamma_1}$. Now, $[e_{-\gamma_1},e_{v_{j+1} - 2(m_j-l_j) +1}] = e_{\delta}$ where 
\begin{align*} \delta = & -\gamma_1 + \alpha_{v_{j+1} - 2(m_j-l_j) + 1} \\
= & -\left( \alpha_{v_{j+1} - 2(m_j - l_j) + 2} +  \cdots + \alpha_{v_{j+1}} + 2\alpha_{v_{j+1}+1} + \cdots + 2\alpha_{n-2} + 2\alpha_{n-1} + 2\alpha_{n}\right) \in \Phi^-.  \end{align*}

Since $[h_-,e]$ is not supported on $\delta$, there must also exist $\gamma_2 \in \Phi^+ \setminus \{\gamma_1\}$ such that $\delta = \zeta - \gamma_2$ for some $\zeta \in J$, and such that $h$ is supported on $-\gamma_2$. Inspection of $\delta$ immediately shows that $\zeta \not \in J_2$. Therefore $\zeta\in J_1$ and $\zeta$ is simple. We have  $\delta-\zeta$ is a root, and so $\zeta=\alpha_{v_{j+1} - 2(m_j-l_j) +1}$ or $\zeta=\alpha_{v_{j+1}}$. The latter is not in $J$ and the former yields $\gamma_2 = \gamma_1$, a contradiction. 

When $G$ is of type $D$ the proof follows exactly the same way as for $B$.\end{proof}

\begin{lemma}\label{l:heeqeA} Theorem \ref{thm:heeqe} holds when $G$ is of type $A$.   
\end{lemma}

\begin{proof}
By Lemma \ref{lem:heeclassadj} we may assume $[X(T): \Z \Phi]$ is divisible by $2$. Let $\SF(e)=((k_i)_{1\leq i \leq r})$. We start by exhibiting solutions to $[h,e]=e]$ in case $n \not \equiv 1 \pmod 4$ or $\SF(e)$ has an odd $k_i$. It suffices to treat the case $G=\SL_{n+1}$, since if we can solve $[h,e]=e$ in $\g$, we can solve it in any central quotient. Recall \[e = \sum_{i=1}^r \sum_{j=1}^{k_i-1} e_{u_i + j}\] is a standard representative. We define  \[ h = \sum_{i=0}^{\lfloor \frac{n-2}{4} \rfloor} (h_{4i+1} + h_{4i+2} )\quad \text{ for } n \equiv 2, 3 \pmod 4, \quad h = \sum_{i=0}^{\frac{n}{4}} (h_{4i+2} + h_{4i+3})  \quad \text{ for } n \equiv 0\pmod{4}.\]
It is routine to use (\ref{cb2}) to check that $[h,e] = e$.

Now assume $n \equiv 1 \pmod 4$ and $\SF(e)$ has an odd $k_i$. Then we may write $e=e'+e''$, where $e'$ and $e''$ are contained in commuting subalgebras $\g'$ and $\g''$ of type $A_{k_i-1}$ and $A_{n-k_i}$, respectively. Both $k_i-1$ and $n-k_i$ are even, so by the previous paragraph, we have $h'$ and $h''$ such that $[h',e'] = e'$ and $[h'',e'']=e''$. Then $h:=h'+h''$ satisfies $[h,e]=e$.

On the other hand, consider the case that $n \equiv 1 \pmod 4$ and $\SF(e)$ has no odd $k_i$. Suppose for a contradiction that $[h,e]=e$. Since $n+1 \equiv 2 \pmod 4$, there is a separable isogeny from $\SL_{n+1}$ to $G$ which differentiates to an isomorphism of Lie algebras, and so it suffices to treat the case $G \cong \SL_{n+1}$. With no odd $k_i$ and $n \equiv 1 \pmod 4$, there must be an odd number of $i$ for which $k_i \equiv 2 \pmod{4}$ and an even number of $i$ for which $k_i \equiv 0 \pmod{4}$. In particular, $r$ is odd.
By Lemma \ref{l:hintorusgeneric}, there is a solution to $[h_0, e] = e$ for $h_0 \in \h$, so let $h_0 = \sum_i \lambda_i h_i$ be a generic element of $\h$. 

The formulas in Section \ref{sec:formulae} yield \[ [h_0,e] = \sum_{i=1}^r \sum_{j=1}^{k_i-1} (\lambda_{u_i+j-1} + \lambda_{u_i+j+1}) e_{u_i + j} ,\] where we set $\lambda_0 = \lambda_{n+1} = 0$. Since $[h_0,e]=e$ for each $i$ we must have $\lambda_{u_i+j-1} + \lambda_{u_i+j+1} = 1$ for all $j \in \{1, \ldots, k_i-1\}$. Every $k_i$ is even, so it follows that $\lambda_{2j} + \lambda_{2j+2} = 1$ for all $j \in \{0, \ldots, \frac{n-1}{2}\}$. But as $n \equiv 1 \pmod 4$ we obtain the contradiction
\[ 1 = \sum_{j=0}^{\frac{n-1}{2}} (\lambda_{2j} + \lambda_{2j+2}) = \lambda_{0} + \lambda_{n+1} = 0+0 = 0.\qedhere \]
\end{proof}

\begin{lemma} \label{l:heeqeD} Theorem \ref{thm:heeqe} holds when $G$ is of type $D$.     
\end{lemma}

\begin{proof}
By Lemma \ref{lem:heeclassadj} we can assume $G\cong\Spin_{2n}$ or $G\cong \HSpin_{2n}$. Suppose $n \not \equiv 2 \pmod 4$.  We exhibit a solution to $[h,e]=e$. It suffices to treat the case  $G\cong\Spin_{2n}$, since this will imply a solution in $\HSpin_{2n}$. Let \[ h = \sum_{i=0}^{\lfloor \frac{n-3}{4} \rfloor}( h_{4i+1} + h_{4i+2}) \quad \text{ if } n \equiv 0, 3 \pmod 4, \quad h = \sum_{i=0}^{\lfloor \frac{n-2}{4} \rfloor} (h_{4i+2} + h_{4i+3})  \quad \text{ if } n \equiv 1\pmod{4}.\]
Using (\ref{cb2}) (the relevant matrix $\A$ is the Cartan matrix of $G$) we find that if $\gamma = \sum_i c_i \alpha_i$ then $[h,e_\gamma]= (\sum_i c_i) e_\gamma=\text{ht}(\gamma)e_\gamma$. But the height of all the roots in $J$ is seen to be odd. It therefore follows that $[h,e] = e$. 

Now let $n \equiv 2 \pmod 4$ and $\SF(e) = ((k_i)_{i=1}^{r},(l_i)_{i=1}^{s},(m_i)_{i=1}^{s})$. First, suppose there is an odd $k_i$ or $m_i$ and let this odd number be $x$. Then $e = e' + e''$ with $e'$ and $e''$ being standard nilpotent elements in commuting subalgebras $\g'$ and $\g''$ of type $D_{x}$ and $D_{n-x}$, respectively. Since both $x$ and $n-x$ are odd, we have established the existence of solutions to $[h',e']=e'$ and $[h'',e'']=e''$ in $\g'$ and $\g''$. Thus we get the solution $h:=h'+h''$ to $[h,e] = e$. 

Now we assume all $k_i$ and $m_i$ are even.  As all $k_i$ and $m_i$ are even and $n \equiv 2 \pmod 4$, it follows that $\alpha_{2j-1} \in J$ for each $j = 1, \ldots, \frac{n}{2} - 1$; and at least one of $\alpha_{n-1}$ and $\alpha_n$ is also in $J$. Recall that a solution to $[h,e]=e$ yields an element $h_0 \in \h$ such that $[h_0,e] = \pi_1(e)$ by Lemma \ref{l:hintorusgeneric}. Let $h_0= \sum_i \lambda_i h_i$ be a general element of $\h$. Since $h_0 \in \h$ and $e_{2j-1} \in \g_1$ we must have $[h_0,e_{2j-1}] = e_{2j-1}$ for each such $j = 1, \ldots, \frac{n}{2} - 1$. 

Now take the case $G=\Spin_{2n}$. Using (\ref{cb2}) we see that $[h_0,e_{2j-1}] = (\lambda_{2j-2} + \lambda_{2j}) e_{2j-1}$ for $j = 2, \ldots, \frac{n}{2} -1$, $[h_0,e_{1}] = \lambda_{2} e_{1}$, $[h_0,e_{n-1}] = \lambda_{n-2} e_{n-1}$ and $[h_0,e_{n}] = \lambda_{n-2} e_{n}$. Therefore $(\lambda_{2j-2} + \lambda_{2j}) = 1$ for $j = 1, \ldots, \frac{n}{2} -1$ and $\lambda_2 = \lambda_{n-2} = 1$. However, $n \equiv 2 \pmod 4$ so  
\[ 1 = \sum_{j=2}^{\frac{n}{2} - 1} (\lambda_{2j-2} + \lambda_{2j}) = \lambda_{2} + \lambda_{n-2} = 1+1 = 0. \]
Thus there is no solution to $[h,e]=e$ when all $k_i$ and $m_j$ are even in the $\Spin_{2n}$ case.

Finally, suppose $G \cong \HSpin_{2n}$. If $s=0$ then there are two classes of nilpotent elements in $\text{Lie}(G)$, which can be distinguished by which Levi subalgebra $\l$ of type $A_{n-1}$ they are contained in; see Remark \ref{rem:twoclassesD}. If $e$ is contained in the Levi with simple roots $\alpha_1, \ldots, \alpha_{n-2}, \alpha_{n}$, then there is a solution to $[h,e]=e$. Indeed, $h_0 = h_1 + \cdots + h_{n-2} + h_n$ works since $[h_i,e_i] = e_i$ for $i = 1, \ldots,n-2, n$ by Equation (\ref{eq:heHSpin2n}). 

For all other classes (with $s=0$ or otherwise) $e$ is supported on $e_{n-1}$. Using Equation (\ref{eq:heHSpin2n}) we see $[h_0,e_{2j-1}] = (\lambda_{2j-1} + \lambda_n) e_{2j-1}$ and so $\lambda_{2j-1} + \lambda_n = 1$ for $j = 1, \ldots, \frac{n}{2} -1$. Also, $[h_0,e_{n-1}] = (\sum_{j=1}^{\frac{n}{2} -1} \lambda_{2j-1}) e_{n-1}$ so $(\sum_{j=1}^{\frac{n}{2} -1} \lambda_{2j-1})$ must also be $1$ if there is to be a solution. However, $n \equiv 2 \pmod 4$ so  $(\frac{n}{2} -1) \equiv 0 \pmod 4$ and thus
\[ 0 = \sum_{j=1}^{\frac{n}{2} - 1} 1 =  \sum_{j=1}^{\frac{n}{2} - 1} (\lambda_{2j-1} + \lambda_{n}) = \left(\sum_{j=1}^{\frac{n}{2} - 1} \lambda_{2j-1}\right) + \left(\frac{n}{2} -1\right) \lambda_{n} = \left(\sum_{j=1}^{\frac{n}{2} - 1} \lambda_{2j-1}\right) . \]
Therefore, there is no solution to $[h_0,e] = \pi_1(e)$ in this case.  
\end{proof}

\begin{lemma}\label{l:heeqeB} Theorem \ref{thm:heeqe} holds for $G \cong \Spin_{2n+1}$.   
\end{lemma}

\begin{proof}
Let  $e$ have specification $\SF(e)=((k_i)_{i=1}^{r},(l_i)_{i=1}^{s},(m_i)_{i=1}^{s},m)$. Recall from Section \ref{BtoD} the embedding of $G$ into $M \cong \Spin_{2n+2}$. Then $e$ is conjugate to an element $e'$ of $\m=\Lie(M)$ with $\SF(e')=((k_i)_{i=1}^{r},(l_i)_{i=1}^{s},(m_i)_{i=1}^{s+1})$ for $m_{s+1} = m$. 

If $n \equiv 0$ or $3 \pmod 4$, then the rank of $M$ is $1$ or $0 \pmod 4$, respectively, and take $h\in\m$ with $[h,e]=e$ from the proof of Lemma \ref{l:heeqeD}. This $h$ is visibly stable under the standard graph automorphism of $M$ and therefore we have $h\in \g$.

When $n \equiv 1 \pmod 4$ and all $k_i$, $m_j$ are even it follows that $m$ is also even. Thus Lemma \ref{l:heeqeD} shows that there is no solution to $[h,e'] = e'$ in $\m$ and hence no solution to $[h,e] = e$ in $\g$. If $n \equiv 2 \pmod 4$, the requirement for at least one $k_i$ or $m_j$ to be odd follows using a very similar argument to that in the proof of Lemma \ref{l:heeqeD}. 

For the remainder of the proof $n \equiv 1, 2 \pmod 4$ and an odd $k_i$ or $m_i$ exists, which we call $x$. It follows from the explanations in Section \ref{sec:standardforms} that $e$ is contained in a subsystem subalgebra $\m_1 + \m_2$ with $\m_1$ of type $D_x$ and $\m_2$ of type $B_{n-x}$. Furthermore, we may assume that $e = y_1 + y_2$ with $y_i \in \m_i$ nilpotent elements in standard form. For ease, we choose bases $\mathcal{B}_1 = \{ \alpha_1, \ldots, \alpha_{x-1}, \alpha_{x-1} + 2 \alpha_x + \cdots + 2\alpha_n\}$ (when $x > 1$) and $\mathcal{B}_2 = \{\alpha_{x+1}, \ldots, \alpha_n\}$ for the subsystems $D_x$ and $B_{n-x}$, respectively. We proceed by finding $h \in \g$ with $[h,y_1] = y_1$ and $[h,y_2] = y_2$ for any standard nilpotent elements $y_i \in \m_i$. Lemma \ref{l:heeqeD} and the previous cases in this proof show the existence of $h_i \in \m_i$ with $[h_i,y_i] = y_i$ whenever $n-x \equiv 0,3 \pmod 4$. Thus the sum $h = h_1 + h_2$ is sufficient unless $(n,x) \equiv (1,3)$ or $(2,1) \pmod 4$.  

For these two cases let
\[ h = \sum\limits_{i=0}^{\frac{x-7}{4}} (h_{4i+2} + h_{4i+3})  + h_{x-1} + h_{x} + h_{x + 1} + \sum\limits_{i=0}^{\frac{n-x-6}{4}} (h_{4i + x+4} + h_{4i + x +5}),\] if $(n,x) \equiv (1,3) \pmod 4$ and \[
h = \sum\limits_{i=0}^{\frac{x-5}{4}} (h_{4i+1} + h_{4i+2})  + h_{x} + \sum\limits_{i=0}^{\frac{n-x-5}{4}} (h_{4i + x+3} + h_{4i + x +4}),\] if $(n,x) \equiv (2,1) \pmod 4$.

We claim $[h,y_i] = y_i$. To see this, first use (\ref{cb2}) to see that if $\delta = \sum_i c_i \alpha_i$ is any root, then 
\begin{equation}\label{e:Bhee} [h,e_\delta]= \sum_{i \neq x} c_i e_\delta.\end{equation}

Therefore, $[h,e_\alpha] = e_\alpha$ for $\alpha \in \mathcal{B}_i$. Inspecting the definitions of the roots $\beta_i \in J_2$ in Table \ref{tab:stdreps} for types $B$ and $D$, we see $[h,e_{\beta}]= e_{\beta}$ for the relevant additional roots $\beta$ in our standard nilpotent elements $y_1 \in \m_1$ and $y_2 \in \m_2$, thus completing the proof. 
\end{proof}

\section{Extending nilpotent elements to \texorpdfstring{$\s$}{s}-subalgebras} \label{sec:fsl2}

The simple $\k$-group $G$ reverts to its general state. 

It is a straightforward exercise to check that there is a unique simple Lie algebra $\s$ of dimension $3$. It has basis $\{e,h,f\}$ subject to the relations \begin{equation}[e,f]=h,\ [h,e]=e,\ [h,f]=f.\label{fslrel}\end{equation} 

\begin{theorem}\label{thm:mainfsl2} Suppose $e$ is a non-zero nilpotent element of $\g$. Then $e$ is contained in a subalgebra $\s$ if and only if
\begin{enumerate}
\item $G$ is of type $A_n$ and $\SF(e)$ has no $k_i \equiv 2 \pmod 4$;
\item $G$ is of type $B_n$ and $\SF(e)$ has $m \not \equiv 2,3 \pmod 4$ and no $k_i, m_i \equiv 2 \pmod 4$;
\item $G$ is of type $C_n$ and $\SF(e)$ has $s=t=0$ and no $k_i \equiv 2 \pmod 4$;
\item $G$ is of type $D_n$ and $\SF(e)$ has no $k_i, m_i \equiv 2 \pmod 4$;
\item $G$ is of type $G_2$ and $e$ has label: $G_{2}$, $G_{2}(a_{1})$;
\item $G$ is of type $F_4$ and $e$ has label: $F_{4}$, $F_{4}(a_{1})$, $F_{4}(a_{2})$, $B_{3}$, $F_{4}(a_{3})$, $(B_{2})^{(2)}$, $\tilde{A}_{2}$, $A_{2}$; 
\item $G$ is of type $E_6$ and $e$ has label: $E_{6}$, $E_{6}(a_{1})$, $D_{5}$, $E_{6}(a_{3})$, $D_{5}(a_{1})$, $D_{4}$, $A_{4}$, $D_{4}(a_{1})$, $A_{3}$, $2A_{2}$, $A_{2}$; 
\item $G$ is of type $E_7$ and $e$ has label: $E_{6}$, $E_{6}(a_{1})$, $A_{6}$, $D_{5}$, $E_{6}(a_{3})$, $A_{4}+A_{2}$, $D_{5}(a_{1})$, $A_{4}$, $A_{3}+A_{2}$, $D_{4}$, $D_{4}(a_{1})$, $2A_{2}$, $A_{3}$, $A_{2}$;
\item $G$ is of type $E_8$ and $e$ has label: $E_{8}$, $E_{8}(a_{1})$, $E_{8}(a_{2})$, $E_{8}(a_{3})$, $E_{8}(a_{4})$, $E_{8}(a_{5})$, $E_{8}(b_{5})$,  $(D_{7})^{(2)}$, $D_{7}$, $E_{8}(a_{6})$, $D_{7}(a_{1})$, $E_{8}(b_{6})$, $A_{7}$, $D_{7}(a_{2})$, $E_{6}$, 
  $D_{5}+A_{2}$, $E_{6}(a_{1})$, $A_{6}$, $E_{8}(a_{7})$, $D_{5}(a_{1})+A_{2}$, $A_{4}+A_{3}$, $D_{5}$, 
  $E_{6}(a_{3})$, $D_{4}+A_{2}$, $A_{4}+A_{2}$, $D_{5}(a_{1})$, $2A_{3}$, $A_{4}$, $A_{3}+A_{2}$, $D_{4}$, 
  $D_{4}(a_{1})$, $2A_{2}$, $A_{3}$, $A_{2}$.
\end{enumerate}
\end{theorem}

We start with some preliminary results about $\s$. Since $\s$ is simple, it has a unique minimal $p$-envelope $\mathfrak{S}$, and there is a unique $[p]$-map on $\mathfrak{S}$, \cite[p97]{SF88}. Computing with an adjoint basis, we have $h^{[2]}=h$, $e^{[2]}=E$ and $f^{[2]}=F$,  where $[E,e]=[E,h]=[F,h]=[F,f]=0$, $[E,f]=e$ and $[F,e]=f$. Furthermore $E^{[2]}=F^{[2]}=0$. Thus $\{E,e,h,f,F\}$ is a basis of $\mathfrak S$ (with a root system of type $BC_1$ according to an appropriate $1$-torus $T\subseteq \Aut\s$). 

\begin{lemma}The nilpotent cone $\N(\mathfrak{S})$ of $\mathfrak{S}$ is the variety $$\{\lambda E+\mu e+\nu h+\rho f + \sigma F\mid \lambda\rho=\mu\nu, \mu\sigma = \nu\rho\}.$$ We have $\N(\s):=\N(\mathfrak{S})\cap\s=\k\cdot e\oplus \k\cdot f$ as a subvariety. Furthermore, $\N(\s)$ identifies with the variety $\mathcal{M}:=\{n\in\s\mid n \in \im(\ad n)^2\}$.\end{lemma}
\begin{proof}If $n=\lambda E+\mu e+\nu h+\rho f + \sigma F$, then according to the basis $\{e,f,h\}$, the element $N:=\ad n$ is represented by the matrix 
  \[\begin{bmatrix}\nu &\mu &\lambda\\
    \rho & 0 & \mu \\
    \sigma & \rho & \nu\end{bmatrix},\] with characteristic polynomial $\chi_N(t)=t^3+(\nu^2+\lambda\sigma)t +\lambda\rho^2+\sigma\mu^2$. Now, $n$ is nilpotent if and only if $\chi_N(t)=t^3$, from which the  descriptions of $\N(\mathfrak{S})$ and $\N(\s)$ follow. If  $n\in \im\ad(n)^2$ then $n\in\s$; so setting $\lambda=\sigma=0$  and squaring $N$ we get
    \[N^2= \begin{bmatrix}\nu^2+\mu\rho & \nu\mu &\mu^2\\
      \nu\rho & 0 & \nu\mu \\
      \rho^2 & \nu\rho &\nu^2+\mu\rho\end{bmatrix}.\]
 It is easily checked that $\N(\s)\subseteq \mathcal{M}$. On the other hand, if there is a vector $v=[a,b,c]^T$ with $N^2 v=[\mu,\nu,\rho]^T$ then resolving along the second coordinate yields the equation $a\nu\rho+c\nu\mu =\nu$. If $\nu\neq 0$ then we have $a\rho+c\mu=1$ and this leads to a contradiction with the other two equations.\end{proof}
  
The existence of a non-zero element $n\in\mathcal{M}$ is actually enough to distinguish $\s$ among $3$-dimensional Lie algebras.

\begin{lemma}\label{recog}
Let $L$ be a Lie algebra over $\k$ of dimension at most $3$. If there exists $0 \neq e \in L$ such that $e \in \im(\ad e)^2$ then $L\cong\s$. 
\end{lemma}

\begin{proof}
Let $f \in L$  satisfy $[e,[e,f]] = e$ and define $h = [e,f]$. Then we claim that $e,f,h$ is a basis for $L$. Clearly $e$ and $f$ are linearly independent. Suppose that $h = \lambda e + \mu f$. Then $e = [e,h] = \mu h$ which implies that $\mu \lambda = 1$ and $\mu^2 = 0$, a contradiction which proves the claim.

Now, \[ [e,[h,f]] = [[e,h],f] + [h,[e,f]] = [e,f] + [h,h] = [e,f].\]
So $[h,f] + f \in \c_L(e)$. To calculate $\c_L(e)$, suppose $x = ae + bf + ch \in \c_L(e)$. Then $[x,e] = bh + ce$, so $b=c=0$ and thus $\c_L(e) = \langle e \rangle$. Hence $[h,f] = f + \lambda e$. 

Setting $e' = e$, $h' = \sqrt{\lambda} e + h$ and $f' = \sqrt{\lambda}h + f$, we see that $(e',h',f' )$ satisfy the relations (\ref{fslrel}).\end{proof}

If $\s\subseteq \g$, then being simple, $\s$ survives in any central quotient $\pi:\g\to \g/\z$. On the other hand if $\s\subseteq\g/\z$, then $\s$ lifts to a subalgebra $\s'\cong\s$ provided $0\to \z\to \pi^{-1}(\s)\to \s\to 0$ splits. The next result, classifying the central extensions of $\s$, guarantees this.
\begin{lemma}\label{h20}
We have $H^2(\s,\k) = 0$.
\end{lemma}

\begin{proof}
We use \cite[Cor.~7.7.3]{Wei94}. Considering $\k$ as a left $\s$-module, the cohomology modules $H^\bullet(\s,\k)$ are the cohomology of the cochain complex $\Hom_k(\bigwedge^\bullet(\s), \k) \cong (\bigwedge^\bullet(\s))^*$.

For $a_1,a_2 \in \s$ and $x \in \{e,f,h\}$ we calculate

\[ \delta_1 x^*(a_1,a_2) = a_1 x^*(a_2) + a_2 x^*(a_1) + x^*([a_1,a_2]) \]
\[ = x^*([a_1,a_2]),  \]
since the action of $\s$ is trivial on $\k$. Thus, $\delta_1 e^*(e,f) = 0, \delta_1 e^*(e,h) = 1$ and $\delta_1 e^*(f,h) = 0$, so $\delta_1 e^* = (e,h)^*$. Similarly, $\delta_1 f^* = (f,h)^*$ and $\delta_1 h^* = (e,f)^*$. We now see that the image of $\delta_1$ spans the $3$-dimensional space $\bigwedge^2(\s)^*$ and hence $H^2(\s,\k) = 0$. 
\end{proof}

Using the above lemma we may reduce to the situation where $G$ is simply connected. We leave the discussion of computations when $G$ is of exceptional type to Section \ref{sec:excep}. Hence, for the remainder of this section, we assume:
\begin{center}\framebox{$G\cong \SL_n$, $\Spin_{2n+1}$, $\Sp_{2n}$ or $\Spin_{2n}$ and $e$ is a representative from Table \ref{tab:stdreps}.}\end{center}
We continue with a useful observation.
\begin{lemma} \label{l:levifenough} Suppose $e = \sum_{j\in J_1} e_j$ is a sum of simple root vectors. Then $e \in \im(\ad e)^2$ if and only if there is a solution to $(\ad e)^2 (\sum_{j \in J_1} a_i e_{-j}) = e$.
\end{lemma}
\begin{proof}
One direction is trivial. For the other, let $f \in \g$ be such that $(\ad e)^2 f = e$. Since $e \in \g_1$ we have $e = \pi_1((\ad e)^2 f) = \pi_{1}((\ad e)^2 (\pi_{-1}(f))) = (\ad e)^2 (\pi_{-1}(f))$. We may write $\pi_{-1}(f) = \sum_{i=1}^{n} a_i e_{-i}$ for some $a_i$. Since $[e_j,e_{-i}]=0$ for $i\neq j$, we get $[e,\pi_{-1}(f)] = \sum_{j\in J} a_j h_j$ as required.
\end{proof}

\begin{lemma} \label{lem:fsl2A} Theorem \ref{thm:mainfsl2} holds for $G = \SL_{n+1}$. \end{lemma}
\begin{proof}
Suppose $\SF(e)=((k_i)_{1\leq i \leq r})$. Then $e$ is a sum of regular nilpotent elements in a subsystem subalgebra $\m$ which is a commuting sum $\m_1+\dots+\m_r$ with each $\m_i$ of type $A_{k_i-1}$ (Table \ref{tab:stdreps}). If $e\in\s$ then there is $f$ such that $(\ad e)^2(f)=e$. By Lemma \ref{l:levifenough}, we may assume $f=\sum_{\alpha\in J} a_\alpha e_{-\alpha}$, an element of $\m$. Such an $f$ exists if and only if it does on projection to each factor $\m_i$ of $\m$. Hence it suffices to treat the case $r=1$ and $e$ is regular.

Theorem \ref{thm:heeqe}(iii) immediately tells us that there is no $\s$ overalgebra of $e$ when $n \equiv 1 \pmod 4$ since \emph{a fortiori} there is no element $h$ such that $[h,e]=e$. Now suppose $n \not \equiv 1 \pmod 4$. Let \[ f = \sum_{i=0}^{\lfloor \frac{n-2}{4} \rfloor} f_{4i+1} + f_{4i+2} \quad \text{ for } n \equiv 2, 3 \pmod 4, \quad f = \sum_{i=0}^{\frac{n}{4}} f_{4i+2} + f_{4i+3}  \quad \text{ for } n \equiv 0\pmod{4}.\]
Using $[e_i,f_j]=\delta_{ij}h_i$, we see $h:=[e,f]$ is the same  as that specified in the proof of Theorem \ref{l:heeqeA}; in particular, $[e,[e,f]]=e$. It is another easy calculation to see $[h,f]=f$, which gives $\langle e,h,f\rangle\cong\s$.
\end{proof}

The type $C$ case now follows easily. 

\begin{lemma} Theorem \ref{thm:mainfsl2} holds when $G \cong \Sp_{2n}$. \end{lemma}
\begin{proof} Suppose $\langle e,h,f\rangle\cong\s\subseteq \g$ with $[h,e]=e$. Theorem \ref{thm:heeqe} yields $\SF(e)=((k_i)_{i=1}^r,(),(),())$. By Table \ref{tab:stdreps}, this means $e=\sum_{j \in J_1} e_j$ is a sum of simple root vectors, such that $\alpha_n\not\in J_1$. Now by Lemma \ref{l:levifenough} there is a solution to $(\ad e)^2 (f) = e$, where $f=\sum_{i=1}^{n-1} a_i f_{i}$.  Thus $f$ is in the type $A_{n-1}$ subalgebra $\m$ corresponding to the first $n-1$ simple roots and $e$ is a standard representative in $\m$. From Lemma \ref{lem:fsl2A} it follows that $\SF(e)$ has no $k_i\cong 2 \pmod 4$. Conversely, any other $\SF(e)$ with $s=t=0$ has a solution for $f$ by the same token.\end{proof}

\begin{lemma} \label{l:Dnfsl2dis}
Let $G \cong \Spin_{2n}$ with $n \not \equiv 2 \pmod 4$ and  $\SF(e)=((n),(),())$ or $((),(l),(n))$ for any $\frac{n + 1}{2} < l \leq n$. Then $e$ is contained in an $\s$-subalgebra of $\g$, and so Theorem \ref{thm:mainfsl2} holds in these cases.
\end{lemma} 
\begin{proof}
In the first case, $e=\sum_{i=1}^{n-1}e_i$ is contained in a subalgebra of type $A_{n-1}$. Now Lemma \ref{lem:fsl2A} yields an $\s$-overalgebra of $e$, since $n-1 \not \equiv 1 \pmod 4$. 

In the second case, Table \ref{tab:dns} exhibits an $\s$-overalgebra of $e$ with structure constants as in (\ref{fslrel}). In each case, it is a routine calculation, using the relations in Section \ref{sec:formulae} to check that $\langle e,h,f \rangle \cong \s$. \end{proof}

\begin{table}\begin{center}\renewcommand{\arraystretch}{1.7}\begin{tabular}{c|c|c|c|c}
  $n$ & $h$ & $f_0$ & $l$ &$f$\\\hline
$0\pmod 4$ & $\sum\limits_{i=0}^{n/4 - 1} h_{4i+1} + h_{4i+2}$ & $\sum\limits_{i=0}^{n/4 - 1} f_{4i+1} + f_{4i+2}$ & even & $f_0$\\
&&& odd & $f_0 + \sum\limits_{k=1}^{n-l} e_{\delta_k}$\\\hline
$1\pmod{4}$ & $\sum\limits_{i=0}^{(n-5)/4} h_{4i+2} + h_{4i+3}$ &  $\sum\limits_{i=0}^{(n-5)/4} f_{4i+2} + f_{4i+3}$ & even & $f_0 + \sum\limits_{k=1}^{n-l} e_{\delta_{k}}$\\
&&& odd & $f_0$\\\hline
$3\pmod{4}$ & $\sum\limits_{i=0}^{(n-3)/4} h_{4i+1} + h_{4i+2}$ & $\sum\limits_{i=0}^{(n-3)/4} f_{4i+1} + f_{4i+2}$ & even or $n$ & $f_0$\\
&&& otherwise & $f_0 + \sum\limits_{k=1}^{n-l} e_{\delta_{k}}$,
  \end{tabular}\end{center}
  \[\text{where}\qquad  \delta_{k} = \sum\limits_{i=2l-n+k}^{n-2} \alpha_i + \sum\limits_{j=n-k+1}^{n} \alpha_j, \quad 1 \leq k \leq n-l.\]
  \caption{$\s$-overalgebras in $D_{n}$ when $\SF(e)=((),(l),(n))$}\label{tab:dns}
\end{table}

We can now deal with the general case. 

\begin{lemma} \label{l:Dnfsl2thm}
Theorem \ref{thm:mainfsl2} holds when $G\cong\Spin_{2n}$. \end{lemma} 
\begin{proof}
Let $\SF(e)=((k_i)_{i=1}^r,(l_i)_{i=1}^s,(m_i)_{i=1}^s)$, so that $e$ is contained in a commuting sum $\m=\m_1+\dots+\m_{r+s}$ of subalgebras of type $A_{k_1-1}, \ldots, A_{k_r-1}$, $D_{m_1}, \ldots, D_{m_s}$. If all $k_i \not \equiv 2 \pmod 4$ and all $m_i \not \equiv 2 \pmod 4$ then Lemmas \ref{lem:fsl2A} and \ref{l:Dnfsl2dis} promise an $\s$-overalgebra $\s_j$ of the projection of $e$ to each $\m_j$. Then the diagonal subalgebra $\s\subseteq\s_1\oplus\dots\oplus \s_{r+s}$ is an $\s$-overalgebra of $e$. 

On the other hand, suppose there is $k_i \equiv 2 \pmod 4$ or $m_i \equiv 2 \pmod 4$ and $e$ has an $\s$-overalgebra. Consider the composition $\varphi:G\to\SO_{2n}\to\SL_{2n}$, where the first map is a central isogeny and the second is inclusion. This gives $\dd\varphi(\s)$ as an $\s$-overalgebra of $e':=\dd\varphi(e)$. The Jordan blocks of $e'$ can be read off from the restriction of the natural $\SL_{2n}$ module $V$ to $\k[e']$ in Table \ref{tab:stdforms}; we find $e'$ has blocks of size $k_i$ or $m_i$ on $V$. This contradicts Lemma \ref{lem:fsl2A}.
\end{proof}

We complete this section with two results covering the type $B$ cases. 

\begin{lemma} \label{l:Bnfsl2dis}
Suppose $G\cong \Spin_{2n+1}$ with $n \equiv 0$ or $3 \pmod 4$ and suppose that $e$ is a regular nilpotent element of $\g$. Then $e$ is contained in an $\s$-subalgebra of $\g$ and Theorem \ref{thm:mainfsl2} holds in this case. 
 \end{lemma} 
\begin{proof}
Since $e$ is regular, we have $\SF(e)=((),(),(),n+1)$. We use the embedding $\varphi:G=M^\tau\to M:= \Spin_{2n+2}$ where $G$ is the stabiliser of the graph automorphism $\tau$ of $M$. Let $V$ be the natural $(2n+2)$-dimensional module for $M$ and let $e':=\dd\varphi(e)$. As explained in Section \ref{BtoD}, $\SF(e)=((),(n+1),(n+1))$. Since $n+1 \equiv 1, 0 \pmod 4$, Table \ref{tab:dns} yields an $\s$-overalgebra of $e'$. Furthermore, in both bases the relevant $f$ is simply $f_0$, which is easily seen to be stable under the graph automorphism and hence contained in $\g$. 
\end{proof}

\begin{lemma} \label{l:Bnfsl2thm}
Theorem \ref{thm:mainfsl2} holds for $G\cong \Spin_{2n+1}$. 
\end{lemma} 
\begin{proof}
The proof is similar to that of Lemma \ref{l:Dnfsl2thm}, so we provide a sketch of the details. Let $\SF(e)=((k_i)_{i=1}^r,(l_j)_{j=1}^s,(m_k)_{k=1}^s,m)$, so that $e$ is contained in a commuting sum $\m=\m_1+\dots+\m_{r+s+1}$ of subalgebras of type $A_{k_1-1}, \ldots, A_{k_r-1}$, $D_{m_1}, \ldots, D_{m_s} B_{m-1}$. 

If all $k_i \not \equiv 2 \pmod 4$ and all $m_i \not \equiv 2 \pmod 4$ then Lemmas \ref{lem:fsl2A} and \ref{l:Dnfsl2dis} promise an $\s$-overalgebra $\s_j$ of the projection of $e$ to each $\m_j$. Similarly, Lemma \ref{l:Bnfsl2dis} promises an $\s$-overalgebra $\s_{r+s+1}$ of the projection of $e$ to $\m_{r+s+1}$ when $m -1 \not \equiv 1,2 \pmod 4$. So when all three congruence conditions hold, the diagonal subalgebra $\s\subseteq\s_1\oplus\dots\oplus \s_{r+s}$ is an $\s$-overalgebra of $e$. 

For the remaining cases consider the composition $\varphi:G\to\SO_{2n+1}\to\SL_{2n+1}$, where the first map is a central isogeny and the second is inclusion. Again $\dd\varphi(\s)$ is an $\s$-overalgebra of $e':=\dd\varphi(e)$. This time we find $e'$ has blocks of size $k_i$, $m_i$, $m$ or $m-1$ on $V$.
\end{proof}

\section{Extending nilpotent elements to \texorpdfstring{$\pgl_2$}{pgl2}-subalgebras} \label{sec:pgl2}

The group $G$ returns as a simple $\k$-algebra and $0\neq e\in \g$ is nilpotent. We determine when $e$ extends to a $\pgl_2$-triple $(e,h,f)$, where $[e,f]=0$, $[h,e]=e$ and $[h,f]=f$. Note that if $e'\in\pgl_2$ is nilpotent then $e'\in \N(\g)=[\pgl_2,\pgl_2]$, and by change of basis we may assume $e=e'$.

\begin{theorem}\label{thm:mainpgl2}
The nonzero element $e$ extends to a $\pgl_2$-triple $(e,h,f)$ if and only if
\begin{enumerate}
\item $G$ is not of type $A_2$ and there exists $h \in \g$ such that $[h,e] = e$ (see Theorem \ref{thm:heeqe});
\item $G$ is of type $A_2$ and $e$ is not regular. 
\end{enumerate}
\end{theorem}

Again we postpone the calculations for exceptional $G$ to Section \ref{sec:excep}. If $e$ extends to a $\pgl_2$-triple then there is $h$ with $[h,e]=e$, so we may assume
\begin{center}\framebox{$G$ is classical, $e$ is a standard representative from Table \ref{tab:stdreps} and $(G,e)$ is listed in Theorem \ref{thm:heeqe}.}\end{center}

\begin{lemma}\label{l:pgl2A} Theorem \ref{thm:mainpgl2} holds when $G$ is of type $A_n$. 
\end{lemma}
\begin{proof}
The result is clear when $n=1$. Suppose $n=2$; then $\g\cong\sl_3$. Suppose also that $e = e_1 + e_2$ is regular and $e$ extends to a $\pgl_2$-triple. For ease of notation, write $e_3=e_{\tilde{\alpha}}$ and $f_3=e_{-\tilde{\alpha}}$. Then an easy calculation yields $C_{\g}(e) = \langle e, e_{3} \rangle$, so without loss of generality, $f=e_{3}$. We have assumed there is some $h= \sum_i a_i e_i + \sum_j b_j h_j + \sum_k c_kf_k$ with $[h,e] = e$ and $[h,f] = f$. Now 
\[ [h,e] = b_2 e_1 + b_1 e_2 + (a_1 + a_2)e_3 + c_3 (f_1 + f_2) + c_1 h_1 + c_2 h_2,\]
\[[h,f] = c_3(h_1 + h_2) + c_2 e_1 + c_1 e_2 + (b_1 +b_2) e_3.\]
The equation $[h,e] = e$ forces $c_i = 0$ for all $i$, $a_1 = a_2$ and $b_1 = b_2 = 1$. But then $[h,f] = 0$, a contradiction that establishes the $n=2$ case.

Now suppose $n \geq 3$. When $[X(T) : \Z\Phi]$ is coprime to $2$ we may assume $G$ is adjoint; so $h = \sum_{i=1}^n h_i$ solves $[h,e] = e$. If $n$ is odd we then define $f = e_{\tilde{\alpha}}$ and check $[e,f]=0$ and $[h,f]=f$. 

When $n$ is even, there are two further cases. If $\alpha_n \not \in J$ then $e$ is contained in a subalgebra $\l$ of type $A_{n-1}$ with $n-1\geq 3$ odd, in which case we have already exhibited a $\pgl_2$-triple for $e$. If $\alpha_n \in J$, choose $f = e_{\alpha_1 + \cdots + \alpha_{n-1}} + e_{\alpha_2 + \cdots + \alpha_n}$. Then $[h,f] = f$ and $[e_{\alpha_1 + \cdots + \alpha_{n-1}},e] = [e_{\alpha_2 + \cdots + \alpha_{n}},e] = e_{\tilde{alpha}}$ so $[e,f]=0$ (noting that $\alpha_1\in J$).

Now suppose $[X(T) : \Z\Phi]$ is divisible by $2$. Then $n\not\equiv 1 \pmod 4$ or $\SF(e)=((k_i)_{i=1}^r)$ admits at least one odd $k_i$. Since $\z(\pgl_2)=0$, it will suffice to construct a $\pgl_2$-triple in $\sl_{n+1}$.

We take $h$ from the proof of Lemma \ref{l:heeqeA}. In each case $[h,e_{\delta}] = \text{ht}(\delta) e_\delta$. Then when $n$ is odd we may take $f = e_{\tilde{\alpha}}$. When $n$ is even we define $f$ as follows and leave the verification to the reader.  
\[ f = 
\begin{cases}
e_{\alpha_1+ \cdots + \alpha_{n-1}}  &\text{ if } \alpha_n \not\in J, \\
e_{\alpha_1+ \cdots + \alpha_{n-1}} + e_{\alpha_2 + \cdots + \alpha_{n}}, &\text{ if } \alpha_n \in J. 
\end{cases}\qedhere
\]
\end{proof}

We may now assume until the end of the section: \begin{center}\framebox{$G$ has type $B$, $C$ or $D$.}\end{center}

\begin{lemma}\label{l:pgl2adj} Theorem \ref{thm:mainpgl2} holds when $G$ is adjoint or $G\cong\SO_{2n}$. 
\end{lemma}
\begin{proof}
Let $h= \sum_i h_i$ and $f = e_{\tilde{\alpha}}$. The proof of Lemma \ref{lem:heeclassadj} shows $[h,e]=e$ and $[e,f] = 0$ since $f \in \z(\g_{+})$. Finally, since $\mathrm{ht}(\tilde{\alpha})$ is odd,  (\ref{cb2}) (or specifically (\ref{eq:heSO2n})) show that $[h,f]=f$.
\end{proof}

\begin{lemma}Theorem \ref{thm:mainpgl2} holds when $G \cong \Sp_{2n}$. 
\end{lemma}
\begin{proof}
Since there is $h$ with $[h,e]=e$, we have $e\in[\Lie(G), \Lie(G)] \cong \Lie(\SO_{2n})$ so we are done by Lemma \ref{l:pgl2adj}.
\end{proof}

\begin{lemma}Theorem \ref{thm:mainpgl2} holds when $G$ is of type $D$. \label{pgl2D}
\end{lemma}
\begin{proof}
It suffices to exhibit a $\pgl_2$-triple for $e$ when $G \cong \Spin_{2n}$, since $\pgl_2$ is centre-free. In that case, set $f = e_{\tilde{\alpha}}$ and $h$ as in the proof of Lemma \ref{l:heeqeD}. We leave the verification to the reader.
\end{proof}

\begin{lemma}Theorem \ref{thm:mainpgl2} holds when $G \cong \Spin_{2n+1}$. 
\end{lemma}
\begin{proof}
First assume $n \equiv 0, 3 \pmod 4$. Choose $h$ as in the proof of Lemma \ref{l:heeqeB}. Define $f = e_{\tilde{\alpha}}$ which satisfies $[h,f]=f$ thanks again to $\tilde{\alpha}$ having an odd height. Furthermore, $f \in Z(\g_{+})$ so $[e,f] = 0$. 

Now suppose $n \equiv 1, 2 \pmod 4$ and there is a solution to $[h,e]=e$. As in the proof of Lemma \ref{l:heeqeB}, we assume $e = y_1 + y_2$ is a sum of standard nilpotent elements contained in a subsystem subalgebra $\m_1 + \m_2$ with $y_1 \in \m_1$ of type $D_x$, $x$ odd and $y_2 \in \m_2$ of type $B_{n-x}$. Furthermore, Lemma \ref{pgl2D} and the previous paragraph yield $\pgl_2$-overalgebras of $y_i \in \m_i$ unless $(n,x) \equiv (1,3)$ or $(2,1) \pmod 4$. In the remaining two cases an $h$ is provided in the proof of Lemma \ref{l:heeqeB} that satisfies $[h,y_i] = y_i$. Thanks to the choice of subsystem subalgebra and $y_i$ being in standard form, $e = y_1 + y_2 \in \g_{\geq 1}$. We break off into two cases. If $x \neq 1$ then set $f = e_{\tilde{\alpha}}$, so $[y_1+y_2,f] = 0$. It follows from Equation (\ref{e:Bhee}) that $[h,f] = f$ showing we have found a $\pgl_2$-subalgebra containing $e$. 

Finally, suppose $x = 1$, so $n \equiv 2 \pmod 4$. Then $e$ is contained entirely within the subalgebra $\m_2$ of type $B_{n-1}$, where $n-1 \equiv 1 \pmod 4$. We have already proven that $e$ has a $\pgl_2$-overalgebra in $\m_2$, completing the proof. 
\end{proof}

\section{Extending nilpotent elements to \texorpdfstring{$\sl_2$}{sl2}-subalgebras} \label{sec:sl2}

The simple $k$-group $G$ returns to arbitrary type. In this section we find when $e\in\N$ can be extended to an $\sl_2$-triple $(e,h,f)$ with 
\[ [e,f] = h, \ \ [h,e] = [h,f] = 0 .\]
In other words, we find when $e$ sits inside a $3$-dimensional Heisenberg algebra. 

\begin{theorem}\label{thm:mainsl2} The nonzero nilpotent element $e$ extends to an $\sl_2$-triple if and only if $G$ is not isomorphic to $\PGL_2$. 
\end{theorem}
\begin{proof}
The case that $G$ has rank $1$ is clear, since $\sl_2$ is not a subalgebra of $\pgl_2$. As usual, exceptional types are postponed to Section \ref{sec:excep} and we are left to deal with classical types of rank at least $2$. 

For each non-zero standard nilpotent representative $e \in \g$ we exhibit an $\sl_2$ overalgebra in $\g_{+}$. Since $\g_+ \cap \z(\g)=0$, the map $\g\to\g/\z(\g)$ induces a bijection on the respective nilpotent cones. Thus we may assume $G$ is simply connected and $e$ is a standard representative supported on roots $J$. Note that since $e$ is non-zero, $\alpha_1 \in J$.  We proceed by exhibiting $f \in \g_+$ such that $\langle e, f \rangle \cong \sl_2$. 

For types $A$ and $C$, define \[ f = e_{\tilde{\alpha}-\alpha_1}.\] Since $\alpha_1 \in J$, it follows that $[e,f] = e_{\tilde{\alpha}}$. Since $e_{\tilde{\alpha}} \in \z(\g_{+})$, we are done. 

For types $B$ and $D$ we differentiate between two cases and define 
\[ f = 
\begin{cases}
e_{\tilde{\alpha}-\alpha_2}  &\text{ if } \alpha_2 \in J, \\
e_{\tilde{\alpha}-\alpha_1 - \alpha_2}, & \text{ if } \alpha_2 \not\in J. 
\end{cases}
\]

Then $[e,f] = e_{\tilde{\alpha}}$, $e_{\tilde{\alpha} - \alpha_1}$, respectively. It is straightforward to check that $e$ and $f$ are in the centraliser of $[e,f]$ in both cases. 
\end{proof}

\section{Main theorems for exceptional groups} \label{sec:excep}
In this last section, we let $G$ be a simple $k$-group of exceptional type. We describe the calculations which establish Theorems \ref{thm:heeqe}, \ref{thm:mainfsl2}, \ref{thm:mainpgl2} and \ref{thm:mainsl2} for $G$, completing their proofs. The calculations were done in Magma, and the code can be found at \url{https://github.com/davistem/nilpotent_orbits_magma}.

\subsection{Nilpotent orbit representatives} Any nilpotent element $e\in\g$ is conjugate to an element in $\g_+$. Since $\g_+$ has the positive root vectors as a basis, each orbit $\O\subseteq\g$ is represented by a linear combination of positive root vectors. 
It turns out that one can always give a representative which is a sum of at most $n+1$ root vectors, where $n$ is the rank of $G$. For a full description of how these are obtained, labelled and verified, see \cite{KST21} or {LS04}.

In practice, we have a function which given an exceptional type and isogeny produces a Lie algebra $\g_{\F_2}$ over $\F_2$ such that $\g_{\F_2}\otimes_{\F_2}\k=\g$, together with representatives of nilpotent orbits and their labels. The algebra $\g_{\F_2}$ comes with a canonical Chevalley basis whose associated structure constants are exactly as described in Section~\ref{sec:present}, and the representatives are as sums of of these basis elements. For example:
\begin{lstlisting}[upquote=true]
 > Read("exceptional_code_functions.m");
 > g,orbs,labels:=Liealgreps("G2","Ad"); 
 > E,F,H:=ChevalleyBasis(g); 
 > e:=orbs[1]; e;
 (0 0 0 0 0 0 0 0 1 1 0 0 0 0)
 > labels[1]; 
 "G_2"
\end{lstlisting}
Even when $\g$ is of type $E_8$ there are only 75 orbits to deal with and the following calculations are all handled in seconds.

\subsection{Solving equations in \texorpdfstring{$\g$}{g}}
For $e\in\g$ nilpotent, there is a solution to $[h,e]=e$ if and only if $e\in\im\ad(e)$. Since $e\in\g_{\F_2}$, the solution set $S:=\{h\in\g\mid [h,e]=e\}$ is stable under $\mathrm{Gal}(k/\F_2)$; thus $[h,e]=e$ has a solution in $\g$ if and only if it has one in $\g_{\F_2}$. Letting $E$ be the matrix of the linear map $\ad(e)$ relative to a given basis of $\g$, we use Magma's linear algebra capabilities to check when $e\in \im E$.

\begin{lstlisting}
 > g,orbs,labels:=Liealgreps("E7","SC");
 > Ad:=AdjointRepresentation(g); 
 > e:=orbs[1];
 > M:=Matrix(Ad(e));
 > e in Image(M); 
 false
\end{lstlisting}

Repeating this process over all representatives of nilpotent orbits, we establish Theorem \ref{thm:heeqe}. 

Theorem \ref{thm:mainfsl2} is similar. We have $e$ contained in an $\s$-subalgebra only if $e\in \im E^2$. We exclude those orbits which fail to have a solution. For those that succeed, Magma provides a solution for $E^2(f)=e$ with $f$ expressed in terms of the given basis. In each case, it turned out  that the $f$ supplied satisfied $[h,f]=f$ and $[h,e]=e$, where $h:=[e,f]$. For example, 

\begin{lstlisting}
  > g,orbs,labels:=Liealgreps("E7","Ad");
  > E,F,H:=ChevalleyBasis(g);
  > V:=VectorSpace(g);
  > Ad:=AdjointRepresentation(g); 
  > e:=orbs[5];
  > M:=Matrix(Ad(e));
  > e in Image(M^2); 
  true
  > v:=Solution(M^2,e);
  > f:=g!(V!v);
  > h:=e*f;
  > h*e eq e;
  true
  > h*f eq f;
  true
 \end{lstlisting}

If $e\in\pgl_2$ then we must have $[h,e]=e$. For each such representative, we are always able to find $f  \in \g_+$ with $[h,f]=f$ and $[e,f]=0$. (Magma looks to use a single positive root vector for $f$ and is always successful.) This deals with Theorem \ref{thm:mainpgl2}.

Lastly, for Theorem \ref{thm:mainsl2}, we ask Magma to look for a root vector $f$ such that $[e,f]\neq 0$ and $[e,f]$ commutes with $e$ and $f$. Happily one is always found.
{\footnotesize
\bibliographystyle{amsalpha}
\bibliography{bib}}

\end{document}